\documentclass[11pt]{article}
\usepackage[utf8]{inputenc}
\usepackage[a4paper,top=2.5cm,bottom=1.5cm,left=1.5cm,right=2cm,marginparwidth=1.5cm]{geometry}
\usepackage{amsmath}
\usepackage{amssymb}
\numberwithin{equation}{section}
\usepackage[english]{babel}
\usepackage{amsthm}
\usepackage{bbm}
\usepackage{amsfonts}
\usepackage{comment}
\usepackage{mathrsfs}
\usepackage{hyperref}
\usepackage{titlesec}
\usepackage[numbers,sort]{natbib}
\hypersetup{
    colorlinks=true,
    linkcolor=blue,
    citecolor=red,
    filecolor=magenta,      
    urlcolor=black,
    pdftitle={5Goup},
    pdfpagemode=FullScreen
    }
\usepackage{graphicx, tikz}
\usepackage{subcaption}
\usepackage{float}
\usepackage{xcolor}
\usepackage{bbm}
\usepackage[format=plain,
            font=it]{caption}

\newtheorem{theorem}{Theorem}[section]

\newtheorem{lemma}[theorem]{Lemma}

\newtheorem*{remark}{Remark}

\newcommand{\upd}{\mathrm{d}}  

\def\a{\alpha }              
     \def\d{\delta}          
    \def\ve{\varepsilon}    \def\z{\zeta }

     \def\t{\tau }

       \def\w{\omega }         \def\W{\Omega }

\title{Reduced-order modeling for electromagnetic inverse problems: a layered medium benchmark}
\author{Konstantinos Alexopoulos\thanks{CMAP, Ecole Polytechnique, Institut Polytechnique de Paris, 91120 Palaiseau, France} \and Josselin Garnier \thanks{CMAP, CNRS, Ecole Polytechnique, Institut Polytechnique de Paris, 91120 Palaiseau, France}}
\date{}

\begin{document}

\maketitle

\begin{abstract}
    We study reduced-order modeling for inverse problems in layered media, focusing on the recovery of impedance profiles from time-domain measurements. Using the Goupillaud structure, we formulate the forward problem as a discrete dynamical system and introduce a ROM-based objective defined at the level of reduced operators. Through numerical experiments on a 5-layer medium under various structured perturbations, we compare the ROM-based objective with a classical data misfit. The results show that the ROM-based inversion provides a more favorable reconstruction in the clean setting and in certain structured perturbation regimes, while remaining consistently competitive with the classical approach across all cases considered.
\end{abstract}


\section{Introduction}

Wave propagation in heterogeneous media plays a central role in many areas of science and engineering, including radar and sonar imaging \cite{cheney2009fundamentals, curlander1991synthetic, blondel1997handbook, gilman2017transionospheric}, seismology and geophysical exploration \cite{symes2008migration, virieux2010overview} and medical imaging via ultrasound \cite{szabo2013diagnostic}. In these applications, one seeks to recover information about an inaccessible medium from measurements of waves generated by controlled sources and recorded by an array of sensors. This leads to inverse problems for wave equations, where the goal is to estimate spatially varying material parameters such as wave speed, density, or electromagnetic coefficients.

Mathematically, wave propagation is typically modeled by second-order wave equations or, more generally, by first-order hyperbolic systems \cite{fouque2007wave, monk2003finite}. The unknown medium properties enter as coefficients in these equations and the measurements correspond to partial observations of the resulting wave field. The associated inverse problem is highly nonlinear and ill-posed, as the mapping from the coefficients to the observed data is complicated and sensitive to perturbations. Fundamental questions of uniqueness and stability have been extensively studied in the literature \cite{belishev2007recent, yu2001global, stefanov2005stable}, but these results often rely on idealized data that are not available in practical applications.

In practice, waveform inversion is commonly formulated as a nonlinear least-squares optimization problem, known as full waveform inversion (FWI), where one minimizes the discrepancy between measured and simulated data. Despite its widespread use, FWI suffers from severe limitations due to the non-convexity of the objective function. In particular, for band-limited and high-frequency data, the objective function exhibits numerous local minima, a phenomenon known as \emph{cycle skipping}, which causes gradient-based optimization methods to fail even when initialized close to the true solution.

Several approaches have been proposed to mitigate these difficulties. One direction consists of replacing the classical $L^2$ data misfit by alternative metrics, such as the Wasserstein distance from optimal transport theory \cite{engquist2013application, yang2018application}. This approach can improve the convexity of the objective function in simple settings \cite{engquist2022optimal}, although this property does not hold in general heterogeneous media \cite{borcea2024data}. Another approach is the so-called modeling operator extension or extended FWI \cite{huang2018source, van2013mitigating, warner2016adaptive}, which introduces additional degrees of freedom in the inversion procedure. While effective in certain configurations, its theoretical guarantees are limited to simplified acquisition geometries \cite{symes2022error}.

A different and promising strategy is based on reduced-order modeling (ROM) \cite{druskin2016direct, tether1970construction, druskin2018nonlinear, borcea2018untangling, borcea2020reduced, borcea2022reduced, borcea2023waveform, borcea20232waveform, borcea2024data}. In this framework, one constructs a low-dimensional dynamical system that captures the essential features of wave propagation in the unknown medium. These ROMs are data-driven: they are computed directly from the measurements without requiring knowledge of the wave field inside the medium. The construction relies on Galerkin projection onto spaces spanned by wave snapshots \cite{benner2015survey, brunton2022data, hesthaven2022reduced}, leading to reduced operators that encode the propagation dynamics. Inversion methods based on ROM have been shown to yield objective functions with improved optimization properties compared to classical FWI. Despite these advances, several challenges remain. In particular, understanding the behavior of ROM-based inversion in structured media and under realistic perturbations is still an open problem. Moreover, while the general ROM framework applies to a wide class of wave equations, its concrete implementation and analysis often benefit from simplified settings that retain the essential features of wave propagation, such as multiple scattering and reflection phenomena.

In this work, we investigate reduced-order modeling for inverse problems in layered media, with a focus on Goupillaud-type structures. This setting provides a natural discrete representation of wave propagation, where scattering events occur at discrete times and the wave field evolves according to a dynamical system. This structure allows for an explicit connection between wave propagation, data acquisition and reduced-order modeling. Our main contributions are as follows. The reduced-order modeling framework that we use follows closely the construction introduced in \cite{BORCEA2024113272, borcea2021reduced}. Our contribution is to adapt and analyze this framework in the specific setting of layered media, where the wave dynamics admit a discrete representation. First, we formulate a ROM-based objective for layered media that compares normalized reduced operators and removes global scaling effects. Then, we establish a precise connection between the Goupillaud structure and a data-driven discrete dynamical system, allowing for an explicit ROM construction. Finally, we provide a systematic numerical study under structured perturbations, identifying regimes in which the ROM-based inversion improves stability and reconstruction accuracy while remaining competitive in all cases.

The paper is organized as follows. In Section \ref{sec:Problem Formulation}, we introduce the electromagnetic forward problem and derive its first-order hyperbolic formulation. In Section \ref{sec:Scalar}, we show how this system can be reduced to a scalar wave equation in a layered medium. Section \ref{sec:Goup} is devoted to the analysis of one-dimensional layered media and the connection with the ROM framework. In Section \ref{sec:StabROM}, we formulate the inverse problem and introduce the ROM-based objective function. Section \ref{sec:NumExp} presents numerical experiments comparing the ROM-based inversion with classical approaches under various perturbations. Finally, Section \ref{sec:Conclusion} concludes the paper with a discussion of the results and perspectives for future work.

\section{Problem formulation}\label{sec:Problem Formulation}

We consider the inverse problem of recovering properties of an electromagnetic medium from measurements recorded by an array of antennas. The propagation of electromagnetic waves is governed by Maxwell's equations. In this section we formulate the forward problem, introduce the frequency-domain representation, and describe the reduced-order modeling framework that will be used for inversion.

\subsection{Maxwell's equations}

Let $\Omega \subset \mathbb{R}^3$ be a compact domain with Lipschitz boundary $\partial\Omega$. For $x\in\W$ and $t\in\mathbb{R}$, the electromagnetic fields consist of the electric field $\mathbf{E}(t,x)$ and the magnetic flux density $\mathbf{B}(t,x)$. We assume that the medium is described by a dielectric permittivity tensor $\boldsymbol{\ve}(x)$ and a scalar magnetic permeability $\mu(x)$. The tensor $\boldsymbol{\ve}(x)$ is assumed to be symmetric and positive definite, while $\mu(x)$ is assumed to be strictly positive. Both coefficients are time independent.

The fields satisfy Maxwell's equations
\begin{align}\label{eq:Maxwell}
    \left\{
    \begin{aligned}
        \nabla \times \mathbf{E}(t,x) + \frac{\partial \mathbf{B}}{\partial t}(t,x) &= -\mathbf{J}_m(t,x), \\
        -\nabla \times \bigl(\mu(x)^{-1}\mathbf{B}(t,x)\bigr) + \boldsymbol{\ve}(x)\frac{\partial \mathbf{E}}{\partial t}(t,x) &= -\mathbf{J}_e(t,x),
    \end{aligned}
    \right.
\end{align}
where $\mathbf{J}_e$ and $\mathbf{J}_m$ denote the electric and magnetic current sources, respectively.

We consider vanishing initial conditions prior to the excitation,
\begin{align}
    \mathbf{E}(t,x) = 0, \qquad \mathbf{B}(t,x)=0, \qquad t \ll 0, \ x\in\mathbb{R}^3.
\end{align}
On the boundary of the truncated domain, we use a standard conservative Maxwell realization. For definiteness, we impose perfectly conducting boundary conditions,
\begin{align}\label{boundary conditions physical}
    \mathbf n(x)\times \mathbf E(t,x)=0,
    \qquad x\in\partial\Omega.
\end{align}

\begin{lemma}\label{lem:hyp prob}
    Introducing the state vector
    \begin{align} \label{def: U}
        \mathbf{U}(t,x) :=
        \begin{pmatrix}
            \mathbf{E}(t,x) \\
            \mathbf{B}(t,x)
        \end{pmatrix},
    \end{align}
    Maxwell's equations can be written as a first-order hyperbolic system of the form 
    \begin{align}\label{eq: Max hyp: 1}
        \left(\mathbf{A}+\frac{\partial}{\partial t}\right)\mathbf{U}(t,x)= \mathbf{S}(t,x),
    \end{align}
    where
    \begin{align}\label{def:A}
        \mathbf{A} =
        \begin{pmatrix}
            0 & -\boldsymbol{\ve}(x)^{-1}\nabla\times \mu(x)^{-1} \\
            \nabla\times & 0
        \end{pmatrix}
    \end{align}
    and
    \begin{align}
        \mathbf{S}(t,x) =  - 
        \begin{pmatrix}
            \boldsymbol{\ve}(x)^{-1}  \mathbf{J}_e (t,x)\\
            \mathbf{J}_m(t,x)
        \end{pmatrix}.
    \end{align}
    Here the notation $\nabla\times\mu^{-1}$ means that the multiplication by $\mu(x)^{-1}$ is applied before taking the curl.
\end{lemma}

\begin{proof}
    See Appendix \ref{app:sec1:hyp prob}.
\end{proof}

\begin{lemma}\label{lemma:hyp prob:2}
    Using the normalized fields
    \begin{align}\label{def: E tilde, B tilde}
        \tilde{\mathbf{E}}(t,x) = \boldsymbol{\ve}(x)^{\frac{1}{2}} \mathbf{E}(t,x),
        \qquad
        \tilde{\mathbf{B}}(t,x) = \mu(x)^{-\frac{1}{2}} \mathbf{B}(t,x),
    \end{align}
    where $\boldsymbol{\ve}(x)^{1/2}$ denotes the symmetric positive definite square root of $\boldsymbol{\ve}(x)$, the system \eqref{eq: Max hyp: 1} becomes
    \begin{align}\label{eq:tildeU}
        \left(\tilde{\mathbf{A}} + \frac{\partial}{\partial t}\right) \tilde{\mathbf{U}}(t,x)
        =
        \tilde{\mathbf{S}}(t,x),
        \qquad t\in\mathbb{R}, \quad x\in\W,
    \end{align}
    with
    \begin{align}
        \tilde{\mathbf{U}}(t,x)
        =
        \begin{pmatrix}
            \tilde{\mathbf{E}}(t,x)\\
            \tilde{\mathbf{B}}(t,x)
        \end{pmatrix}.
    \end{align}
    The differential operator $\tilde{\mathbf{A}}$ is
    \begin{align}\label{def: tilde A}
        \tilde{\mathbf{A}} = 
        \begin{pmatrix}
            0 & -\boldsymbol{\ve}(x)^{-\frac12}\nabla\times \mu(x)^{-\frac12}\\
            \mu(x)^{-\frac12}\nabla\times \boldsymbol{\ve}(x)^{-\frac12} & 0
        \end{pmatrix},
    \end{align}
    and
    \begin{align}
        \tilde{\mathbf{S}}(t,x)
        =
        -
        \begin{pmatrix}
            \boldsymbol{\ve}(x)^{-\frac12}\mathbf{J}_e(t,x)\\
            \mu(x)^{-\frac12}\mathbf{J}_m(t,x)
        \end{pmatrix}.
    \end{align}
    In \eqref{def: tilde A}, the coefficient $\mu(x)^{-1/2}$ is part of the argument of the curl. Thus
    \begin{align*}
        \nabla\times \mu(x)^{-1/2}\mathbf v
        =
        \nabla\times\bigl(\mu(x)^{-1/2}\mathbf v\bigr).
    \end{align*}
\end{lemma}

\begin{proof}
    See Appendix \ref{app:sec1:hyp prob:2}.
\end{proof}

\begin{lemma}\label{lem: tilde A skew sym}
    The operator matrix $\tilde{\mathbf A}$ is skew-symmetric on the above domain. With the standard perfectly conducting Maxwell realization, it is skew-adjoint.
\end{lemma}

\begin{proof}
    See Appendix \ref{app:sec1:tilde A skew sym}.
\end{proof}

\subsection{Frequency-domain formulation}

\subsubsection{Excitation expression}

We assume that the data for the inversion are collected by an array of $m$ antennas. Let $p=1,2,\dots,6$ denote the polarizations. Then the excitation of the $s^{\text{th}}$ antenna is given by
\begin{align}\label{def: J_e s,p}
    \mathbf{J}_e^{(s,p)} = 
    \begin{cases}
        \boldsymbol{\ve}^{\frac{1}{2}}(x) f(t) F_e^{(s)}(x) e_p, \quad &s=1,\dots,m, \quad p=1,2,3, \\
        0, &s=1,\dots,m, \quad p=4,5,6,
    \end{cases}
\end{align}
and
\begin{align}\label{def: J_m s,p}
    \mathbf{J}_m^{(s,p)} = 
    \begin{cases}
        0, \quad &s=1,\dots,m, \quad p=1,2,3,\\
        \mu(x)^{\frac{1}{2}} f(t) F_m^{(s)}(x) e_{p-3}, \quad &s=1,\dots,m, \quad p=4,5,6.
    \end{cases}
\end{align}
This choice ensures that the normalized source has the same form in all polarizations. Here $e_1,e_2,e_3$ are the canonical basis vectors in $\mathbb{R}^3$.

Here, $f$ denotes the probing pulses emitted by the antennas. They are supported in a short time interval $(-T_f,T_f)$ and have Fourier transform $\hat{f}$ that is non-negligible at frequencies $\w$ satisfying $|\w\pm\w_0|<\pi b$. The pulse satisfies
\begin{align}
    \hat{f}(0) = \int_{-T_f}^{T_f} f(t) \upd t = 0.
\end{align}

\subsubsection{Component-wise problem formulation}

Using the polarization of each antenna, we write
\begin{align}\label{def: U s,p}
    \tilde{\mathbf{U}}^{(s,p)}(t,x) = 
    \begin{pmatrix}
        \tilde{\mathbf{E}}^{(s,p)}(t,x) \\
        \tilde{\mathbf{B}}^{(s,p)}(t,x)
    \end{pmatrix},
    \quad s=1,\dots,m, \quad p=1,2,\dots ,6.
\end{align}
Substituting \eqref{def: J_e s,p}, \eqref{def: J_m s,p} and \eqref{def: U s,p} into \eqref{eq:tildeU}, we get
\begin{align*}
    \left(\tilde{\mathbf{A}} + \frac{\partial}{\partial t}\right) \tilde{\mathbf{U}}^{(s,p)}(t,x) = - 
    \begin{pmatrix}
        \boldsymbol{\ve}(x)^{-\frac{1}{2}} \mathbf{J}_e^{(s,p)}(t,x) \\
        \mu(x)^{-\frac{1}{2}} \mathbf{J}_m^{(s,p)}(t,x)
    \end{pmatrix}.
\end{align*}
Let $\{\tilde{e}_p\}_{p=1}^6$ be the canonical basis in $\mathbb{R}^6$ and introduce the antenna excitation matrix
\begin{align}
    \mathbf{F}^{(s)}(x) :=
    \begin{pmatrix}
        F_e^{(s)}(x) \textbf{I}_3 & 0\\
        0 & F_m^{(s)}(x) \textbf{I}_3
    \end{pmatrix}
    \in \mathbb R^{6\times 6}.
\end{align}
Then the component-wise system becomes
\begin{align}\label{eq: Maxwell: 1 hyp: 3: component}
    \left(\tilde{\mathbf{A}}+\frac{\partial}{\partial t}\right)\tilde{\mathbf{U}}^{(s,p)}(t,x)
    =
    -f(t) \mathbf{F}^{(s)}(x)\widetilde e_p.
\end{align}
The associated initial condition is
\begin{align}\label{initial condition}
    \tilde{\mathbf{U}}^{(s,p)}(t,x) = \mathbf{0}, \quad t<-T_f, \quad x\in\W.
\end{align}
The boundary conditions are chosen so that the boundary contribution in the curl Green formula vanishes. For the perfectly conducting realization used here, this corresponds in the normalized variables to
\begin{align*} 
    \mathbf n\times(\boldsymbol{\ve}^{-1/2}\tilde{\mathbf E}^{(s,p)})=0.
\end{align*}

The natural domain of $\tilde{\mathbf A}$ is therefore
\begin{align}
\begin{aligned}
    \mathbf{\mathcal{D}} :=  \Bigg\{ \mathbf{\psi}(x) := \begin{pmatrix}
        \psi_1(x)\\ \psi_2(x)
    \end{pmatrix} : \ &
    \nabla \times \Big(\boldsymbol{\ve}(x)^{-\frac12}\psi_1(x)\Big)\in L^2(\W),
    \quad \nabla \times \Big(\mu(x)^{-\frac12}\psi_2(x)\Big)\in L^2(\W),
    \\
    &
    \mathbf{n} \times \Big(\boldsymbol{\ve}^{-\frac12}\psi_1\Big)\big|_{\partial\W}=0,
    \Bigg\}.
\end{aligned}
\end{align}

\subsubsection{Wave decomposition}

For the potential representation of the null-space component, we assume for simplicity that $\Omega$ is simply connected with connected boundary. 

\begin{theorem}[Helmholtz decomposition]\label{thm:Helmholtz dec}
    The solution to \eqref{eq: Maxwell: 1 hyp: 3: component} with the initial and boundary conditions above has the Helmholtz decomposition
    \begin{align}\label{Helmholtz decomposition}
        \tilde{\mathbf{U}}^{(s,p)}(t,x)
        =
        \widetilde{\nabla} \mathbf{N}^{(s,p)}(t,x)
        +
        \mathbb{U}^{(s,p)}(t,x),
    \end{align}
    where $\widetilde{\nabla}\mathbf{N}^{(s,p)}\in\text{null}(\tilde{\mathbf{A}})$ and $\mathbb{U}^{(s,p)}\in \overline{\mathrm{range}(\tilde{\mathbf A})} = \mathrm{null}(\tilde{\mathbf A})^\perp$, for $s=1,\dots,m$ and $p=1,2,\dots,6$. The operator $\widetilde{\nabla}$ is defined by
    \begin{align}\label{eq:nablatilde}
        \widetilde{\nabla} \psi (x) =
        \begin{pmatrix}
            \boldsymbol{\ve}(x)^{\frac12}\nabla \psi_1(x) \\
            \mu(x)^{\frac12}\nabla \psi_2(x)
        \end{pmatrix},
        \quad
        \psi =
        \begin{pmatrix}
            \psi_1\\
            \psi_2
        \end{pmatrix}.
    \end{align}
\end{theorem}

\begin{proof}
    See Appendix \ref{app:sec1:thm:Helmholtz dec}.
\end{proof}

\begin{remark}
    The propagating component $\mathbb{U}^{(s,p)}$ is orthogonal to the null space of $\tilde{\mathbf A}$ and satisfies the corresponding weak divergence constraints. This is the analogue of the standard Maxwell decomposition; see, for instance, Lemma 4.5 in \cite{monk2003finite}.
\end{remark}

Let $P_0$ denote the orthogonal projection onto
$\mathrm{null}(\tilde{\mathbf A})$, and let $P_r=I-P_0$. For
\begin{align*}
    \mathbf F^{(s,p)}(x):=\mathbf F^{(s)}(x)\tilde e_p,
\end{align*}
we write
\begin{align*}
    \mathbf F_0^{(s,p)}:=P_0\mathbf F^{(s,p)},
    \qquad
    \mathbb F^{(s,p)}:=P_r\mathbf F^{(s,p)}.
\end{align*}
The null-space component satisfies
\begin{align}
    \frac{\partial}{\partial t}
    \widetilde{\nabla}\mathbf N^{(s,p)}(t,x)
    =
    -f(t)\mathbf F_0^{(s,p)}(x).
\end{align}
The range component satisfies
\begin{align}\label{eq: U div free}
    \left(\tilde{\mathbf{A}} + \frac{\partial}{\partial t}\right) \mathbb{U}^{(s,p)}(t,x) = - 
        f(t) \mathbb{F}^{(s,p)}(x),
    \quad t\in\mathbb{R}, \ x\in\W,
\end{align}
where $\mathbb{F}^{(s,p)}$ is the range component of $\mathbf{F}^{(s,p)}(x)$.

\subsubsection{Array response matrix}

Let us assume that the antenna profiles $\mathbf{F}^{(s)}$ are spatially localized in a small neighborhood of the receiver position $x_s$. More precisely, each $\mathbf{F}^{(s)}$ is supported in a region whose diameter is small compared with the wavelength and satisfies the normalization condition
\begin{align*}
    \int_{\Omega} \mathbf{F}^{(s)}(x)\upd x = \mathbf{I}_6.
\end{align*}
In this sense, the functions $F^{(s)}$ act as approximate Dirac distributions centered at $x_s$. As a result, the measurements correspond to localized averages of the wave field, and we use the approximation
\begin{align*}
    \int_{\Omega} \mathbf{F}^{(s')}(x) \mathbf{U}(t,x)\upd x \approx \mathbf{U}(t,x_{s'}).
\end{align*}
We define the array response matrix to be the $6m\times6m$ time dependent matrix $\mathcal{W}$. In the $(s,p)$ column of this matrix, we store the recording of $\tilde{\mathbf{U}}^{(s,p)}$ at all the antennas. Then,
\begin{align}\label{def: array response}
    \mathcal{W}^{(s',p'),(s,p)}(t)
    =
    \tilde{e}^{\top}_{p'}
    \int_{\W}
    \mathbf{F}^{(s')}(x)
    \tilde{\mathbf{U}}^{(s,p)}(t,x)
    \upd x
    \approx
    \tilde{e}^{\top}_{p'}
    \tilde{\mathbf{U}}^{(s,p)}(t,x_{s'}).
\end{align}

Let us emphasize that, although the general formulation allows for anisotropic media and variable permeability, we assume that the medium is known, isotropic and homogeneous in the vicinity of the array. In this region, $\boldsymbol{\ve}=\ve_0\mathbf I$ and $\mu=\mu_0$, so the reference scalar wave speed is
\begin{align*}
    c_0=\frac{1}{\sqrt{\mu_0\ve_0}}.
\end{align*}
By vicinity, we mean within a distance $2c_0T_f$ traveled by the waves over the duration of the probing pulse $f$.

Thus, since the medium is known and homogeneous within a distance $2c_0T_f$ from the array, we can compute $\tilde{\mathbf{U}}^{(s,p)}$ for $t\in(-T_f,T_f)$ and remove the component lying in the null space. Then, we obtain the transformed array response matrix $\mathbb{W}(t)\in\mathbb{R}^{6m\times6m}$, with entries
\begin{align}\label{def: W}
    \mathbb{W}^{(s',p'),(s,p)}(t)
    =
    \tilde{e}^{\top}_{p'}
    \int_{\W}
    \mathbf{F}^{(s')}(x)
    \mathbb{U}^{(s,p)}(t,x)
    \upd x
    \approx
    \tilde{e}^{\top}_{p'}
    \mathbb{U}^{(s,p)}(t,x_{s'}).
\end{align}

\subsubsection{Spectral decomposition}

Since we have removed the null space components of $\tilde{\mathbf{A}}$, let
\begin{align}
    \mathcal H_r
    :=
    \mathrm{null}(\tilde{\mathbf A})^\perp
    =
    \overline{\mathrm{range}(\tilde{\mathbf A})}
\end{align}
be the propagating subspace. We denote by
\begin{align}
    \mathbb A
    :=
    \tilde{\mathbf A}\big|_{\mathcal D(\tilde{\mathbf A})\cap \mathcal H_r}
\end{align}
the restriction of the Maxwell operator to this subspace.
The operator $\mathbb{A}$ is skew-adjoint on the projected range space. Hence $-\mathbb{A}^2$ is positive self-adjoint. Under the Maxwell compactness property for the chosen realization, the operator $-\mathbb{A}^2$ has compact resolvent on $\mathcal H_r$; see, e.g., \cite{monk2003finite}. Therefore there exists an orthonormal basis $(\boldsymbol{\phi}_j)_{j\geq 1}$ of the considered space and a sequence of strictly positive eigenvalues $(\lambda_j)_{j\geq 1}$, possibly with multiplicity, such that
\begin{align}
    -\mathbb{A}^2 \boldsymbol{\phi}_j
    =
    \lambda_j \boldsymbol{\phi}_j,
    \qquad
    \lambda_j > 0,
    \quad
    \lambda_j \to \infty.
\end{align}
We then define
\begin{align}
    \boldsymbol{\psi}_j
    :=
    \frac{\mathbb{A} \boldsymbol{\phi}_j}{\sqrt{\lambda_j}}.
\end{align}
It follows that for each $j\geq1$, the pair $(\boldsymbol{\phi}_j,\boldsymbol{\psi}_j)$ spans a two-dimensional invariant subspace of $\mathbb{A}$, and satisfies
\begin{align}\label{eq:A eigs}
    \mathbb{A} \boldsymbol{\phi}_j = \sqrt{\lambda_j}\,\boldsymbol{\psi}_j,
    \qquad 
    \mathbb{A} \boldsymbol{\psi}_j = -\sqrt{\lambda_j}\,\boldsymbol{\phi}_j.
\end{align}
In the case of multiple eigenvalues, the construction is understood within each eigenspace. For the full proof, we refer to Appendix \ref{app:sec2:claim}.

For initial data expanded in the $\boldsymbol{\phi}_j$ components,
\begin{align*}
    v_0(x)=\sum_{j\geq1}\alpha_j\boldsymbol{\phi}_j(x),
\end{align*}
the solution to the homogeneous problem
\begin{align}
    \left(\mathbb{A} + \frac{\partial}{\partial t}\right) v(t,x) = 0
\end{align}
is
\begin{align}\label{sin-cos formula}
    v(t,x)
    =
    \sum_{j\geq 1}
    \alpha_j
    \left(
    \cos(\sqrt{\lambda_j} t)\,\boldsymbol{\phi}_j(x)
    -
    \sin(\sqrt{\lambda_j} t)\,\boldsymbol{\psi}_j(x)
    \right),
\end{align}
where
\begin{align}
    \alpha_j
    =
    \int_{\W} \boldsymbol{\phi}_j(x)^{\top} v_0(x)\upd x.
\end{align}

We now return to the forced problem. The solution $\mathbb{U}^{(s,p)}$ can be written as a convolution in time between the source term and the homogeneous propagator. In the same spirit as in \cite{BORCEA2024113272}, this can be interpreted as mapping the source excitation to an effective initial condition.

\begin{remark}
    Since $\mathbb{A}$ is skew-adjoint on the projected range space, $-\mathbb{A}^2$ is positive self-adjoint. We therefore define the source filter through the functional calculus of the positive operator $\Lambda=(-\mathbb{A}^2)^{1/2}$, rather than through $\mathbb{A}$ itself.
\end{remark}

More precisely, we write
\begin{align}\label{eq:causality}
    \mathbb{U}^{(s,p)}(t,x)
    =
    \Big| \hat{f}(\Lambda) \Big|\, v^{(s,p)}(t,x),
    \qquad
    \Lambda=(-\mathbb{A}^2)^{1/2},
\end{align}
where
\begin{align}\label{def: v}
    v^{(s,p)}(t,x)
    :=
    e^{-t\mathbb{A}} v_0^{(s,p)}(x)
\end{align}
solves the homogeneous problem.

Equivalently, using the spectral decomposition of $-\mathbb{A}^2$,
\begin{align}
    \mathbb{U}^{(s,p)}(t,x)
    =
    \sum_{j\geq 1}
    \left|\hat f(\sqrt{\lambda_j})\right|
    \alpha_j^{(s,p)}
    \left[
    \cos(\sqrt{\lambda_j}t)\boldsymbol{\phi}_j(x)
    -
    \sin(\sqrt{\lambda_j}t)\boldsymbol{\psi}_j(x)
    \right],
\end{align}
with
\begin{align}
    \alpha_j^{(s,p)}
    =
    \int_{\W}
    \boldsymbol{\phi}_j(x')^{\top}
    \mathbf{F}^{(s)}(x')\tilde e_p
    \upd x'.
\end{align}

\begin{remark}
    Since $\mathbb{A}$ is skew-adjoint, Stone's theorem implies that $\mathbb{A}$ generates a strongly continuous unitary group $(e^{-t\mathbb{A}})_{t\in\mathbb{R}}$. Thus, the notation $e^{-t\mathbb{A}}v_0$ is understood in this semigroup sense. On each two-dimensional invariant subspace associated with an eigenfunction of $-\mathbb{A}^2$, this unitary evolution is represented by \eqref{sin-cos formula}.
\end{remark}

\begin{remark}
    The representation \eqref{eq:causality} should be understood as the source-filtered formulation used in the ROM construction, following the data transformation in \cite{BORCEA2024113272}. It is not meant as a pointwise identity for the raw causal field before processing. Rather, after the source filtering and time symmetrization of the measured response, the probing pulse is represented by the positive spectral multiplier $|\hat f(\Lambda)|$, while $v^{(s,p)}$ denotes the corresponding homogeneous evolution.
\end{remark}

The effective initial condition is
\begin{align}
    v_0^{(s,p)}(x)
    :=
    \sum_{j\geq 1}
    \left[
    \int_{\W}
    \boldsymbol{\phi}_j(x')^{\top}
    \mathbf{F}^{(s)}(x')\tilde e_p
    \upd x'
    \right]
    \boldsymbol{\phi}_j(x).
\end{align}

Following the data transformation in \cite{BORCEA2024113272}, we now work with the processed array response, obtained from the measured response by source filtering and time symmetrization. This transformation absorbs the effect of the probing pulse into the effective initial condition and allows the data to be written as inner products of homogeneous wave fields. We denote the transformed data matrix by $\mathbb{D}(t)$:
\begin{align}
    \mathbb{D}^{(s',p'),(s,p)}(t)
    :=
    \tilde{e}^{\top}_{p'}
    \int_{\W}
    \mathbf{F}^{(s')}(x)
    \mathbb{U}^{(s,p)}(t,x)
    \upd x.
\end{align}
With the above definition, the transformed data satisfy
\begin{align}\label{eq: D}
    \mathbb{D}^{(s',p'),(s,p)}(t)
    =
    \int_{\W}
    \Big[
    v_0^{(s',p')}(x)
    \Big]^{\top}
    v^{(s,p)}(t,x)
    \upd x.
\end{align}
This data transformation follows the ROM framework developed for wave-based inverse scattering in \cite{borcea2018untangling, borcea2020reduced, borcea2022reduced, borcea2023waveform, borcea2024data}.

\subsection{ROM formulation}

The reduced-order modeling framework used here follows the general approach developed in \cite{BORCEA2024113272}.

\subsubsection{Wave snapshots}

We are interested in the evolution of the wave fields \eqref{def: v} on a uniform time grid given by
\begin{align*}
    \{ t_j = j\t, \quad j\geq0\},
\end{align*}
where the step is chosen according to the Nyquist criterion $\t\lesssim\frac{\pi}{\w_0}$. In order to ease the notation, we define the $6\times6m$ dimensional fields $v_j$, called the wave snapshot at time $t_j$, given by
\begin{align}\label{eq: v_j = (...)}
    v_j(x) = \Big( v^{(1,1)}(t_j,x), \dots, v^{(1,6)}(t_j,x), \dots, v^{(m,1)}(t_j,x), \dots, v^{(m,6)}(t_j,x) \Big).
\end{align}
Since $\mathbb{A}$ is skew-adjoint on the propagating subspace, Stone's theorem implies that it generates a unitary group $(e^{-t\mathbb{A}})_{t\in\mathbb{R}}$. Hence, from \eqref{def: v},
\begin{align}\label{def: v_j}
    v_j(x)=e^{-j\t\mathbb{A}}v_0(x),
\end{align}
for $x\in\W$ and $j\geq0$. In particular, we know that
\begin{align*}
    e^{(j+1)\a} = e^{\a} e^{j\a}, 
\end{align*}
for $\a\in\mathbb{C}$. Thus, we obtain
\begin{align}\label{def: v_j+1}
    v_{j+1}(x) = \mathcal{P}v_j(x), \quad j\geq0, \quad x\in\W,
\end{align}
where $\mathcal{P}$ is the propagator operator, given by
\begin{align}
    \mathcal{P} = e^{-\t\mathbb{A}}.
\end{align}
Thus, \eqref{def: v_j+1} defines a discrete-time dynamical system with initial state $v_0$.

The ROM is an algebraic discrete dynamical system, obtained via the Galerkin projection of \eqref{def: v_j+1} on the space
\begin{align}
    \mathscr{S} := \mathrm{span}\{ \mathcal{V}(x) \}, \quad \mathcal{V}(x) = \Big( v_0(x), \dots, v_{n-1}(x) \Big), \quad x\in\W.
\end{align}
The end time index $n$ is chosen according to the distance from the array at which we wish to image. If the characteristic length of the medium is $L$, then we should have $n c_0 \t \gtrsim 2L$. This construction is standard in data-driven ROM approaches based on wave snapshots and Galerkin projections \cite{BORCEA2024113272}.

\begin{remark}
    The condition $nc_0\tau \gtrsim 2L$ ensures that the observation time window is sufficiently long for waves to propagate across the medium and return to the array. In particular, this guarantees that the collected snapshots contain information about reflections from the entire domain. If this condition is not satisfied, deeper parts of the medium are not adequately probed and the resulting reduced-order model may fail to capture the relevant wave interactions.
\end{remark}

\subsubsection{Data-driven ROM}

The Galerkin approximation of \eqref{def: v_j} is defined by
\begin{align}\label{def: Galerkin}
    v_j^{\mathrm{Gal}}(x) = \mathcal{V}(x) \mathfrak{g}_j, \quad j\geq0, \quad x\in\W,
\end{align}
where $\mathbf{\mathfrak{g}}_j\in\mathbb{R}^{6mn\times6m}$ are the Galerkin coefficients, calculated so that when substituting \eqref{def: Galerkin} in \eqref{def: v_j+1}, the residual is orthogonal to the space $\mathscr{S}$.

In order to obtain the Galerkin equation, we use an orthonormal basis $\mathcal{Z}(x)$, given by
\begin{align*}
    \mathcal{Z}(x) = \Big(z_0(x),\dots,z_{n-1}(x)\Big),
\end{align*}
whose components satisfy
\begin{align}
    \int_{\W} z_j^{\top}(x) z_k(x) \upd x = \mathbf{I}_{6m} \d_{jk}, \quad j,k=0,\dots,n-1.
\end{align}
Since our basis is causal, we have
\begin{align}
    z_j(x) \in \mathrm{span}\Big\{ v_0(x), \cdots, v_j(x) \Big\}, \quad j=0,\dots,n-1.
\end{align}
Explicitly, $\mathcal{Z}$ is defined via the Gram-Schmidt orthogonalisation of $\mathcal{V}$, i.e.
\begin{align}\label{Gram Schmidt}
    \mathcal{V}(x) = \mathcal{Z}(x) R, \quad x\in\W,
\end{align}
where $R$ is block upper triangular.

The equation for the Galerkin coefficients can now be written as
\begin{align}\label{Galerki coefficients}
    \begin{aligned}
        0 &= \int_{\W} \mathcal{Z}^{\top}(x) \Big[ \mathcal{V}(x)\mathfrak{g}_{j+1} - \mathcal{P} \mathcal{V}(x) \mathfrak{g}_j \Big] \upd x \\
        &= R^{-\top} \Big[ \mathbb{M} \mathfrak{g}_{j+1} - \mathbb{S} \mathfrak{g}_{j} \Big], \quad j\geq0,
    \end{aligned}
\end{align}
with initial condition
\begin{align}
    \mathfrak{g}_{0} = i_0 \ \text{ such that } \ v_0^{\mathrm{Gal}}(x) = v_0(x), \quad x\in\W,
\end{align}
where we have introduced the mass matrix $\mathbb{M}$, given by
\begin{align}\label{def: mass matrix}
    \mathbb{M} := \int_{\W} \mathcal{V}^{\top}(x) \mathcal{V}(x) \upd x \in \mathbb{R}^{6nm\times6nm},
\end{align}
and the stiffness matrix $\mathbb{S}$, given by
\begin{align}\label{def: stiffness matrix}
    \mathbb{S} := \int_{\W} \mathcal{V}^{\top}(x) \mathcal{P} \mathcal{V}(x) \upd x \in \mathbb{R}^{6nm\times6nm},
\end{align}
of the Galerkin scheme. This construction is the standard data-driven ROM construction based on snapshot Gram matrices; see, for example, \cite{druskin2016direct, druskin2018nonlinear, borcea2020reduced, borcea2022reduced}. The following results hold for the mass and stiffness matrices.

\begin{lemma}
    It can be shown that
    \begin{align}\label{M=R^T R}
        \mathbb{M} = R^{\top} R
    \end{align}
    and that the first $n$ Galerkin coefficients are the $6nm\times6m$ block columns of $I_{6nm}$,
    \begin{align}\label{g_j=i_j}
        \mathfrak{g}_j = i_j, \quad j=0,\dots,n-1,
    \end{align}
    where $i_j$ denotes the $j$-th block column of $I_{6mn}$.
\end{lemma}
\begin{proof}
    The proof follows the same logic as in \cite{BORCEA2024113272}.
\end{proof}

\begin{lemma}\label{lemma:M and S properties}
    We can show that both the mass and the stiffness matrix are data driven, i.e.
    \begin{align}\label{data driven mass matrix}
        \mathbb{M}_{j,l} = 
        \begin{cases}
            \mathbb{D}(t_{l-j}), \quad &0\leq j \leq l \leq n-1,\\
            \mathbb{D}(t_{j-l})^{\top}, \quad &0\leq l \leq j \leq n-1
        \end{cases}
    \end{align}
    and
    \begin{align}\label{data driven stiffness matrix}
        \mathbb{S}_{j,l} = 
        \begin{cases}
            \mathbb{D}(t_{l+1-j}), \quad &0\leq j \leq l \leq n-1,\\
            \mathbb{D}(t_{j-1-l})^{\top}, \quad &0\leq l \leq j-1, \text{ with } j=2,\dots,n-1.
        \end{cases}
    \end{align}
\end{lemma}

\begin{proof}
    See Appendix \ref{app:sec1:lemma:M and S properties}.
\end{proof}

These identities show that the ROM can be constructed directly from the data, a key feature of the approach introduced in \cite{borcea2024data}.

\subsubsection{The ROM}

We define the ROM to be the discrete time analog of \eqref{def: v_j+1}, with states
\begin{align}\label{ROM}
    v_j^{\mathrm{ROM}} = R \mathfrak{g}_j, \quad j\geq0,
\end{align}
called the ROM snapshots. The first $n$ such snapshots are the $6nm\times6m$ block columns of $R$, as follows from \eqref{g_j=i_j},
\begin{align}
    \Big(v_0^{\text{ROM}},\dots,v_{n-1}^{\text{ROM}}\Big) = R.
\end{align}

\begin{theorem}
    The ROM snapshots evolve as follows
    \begin{align}\label{def:ROM}
        v_{j+1}^{\mathrm{ROM}} = \mathcal{P}^{\mathrm{ROM}} v_j^{\mathrm{ROM}},
    \end{align}
    where
    \begin{align}\label{P_ROM = R-T S R-1}
        \mathcal{P}^{\mathrm{ROM}} := R^{-\top} \mathbb{S} R^{-1} 
    \end{align}
is the $6nm\times6nm$ ROM propagator matrix.
\end{theorem}

\begin{proof}
    In order to obtain this, we see from \eqref{Galerki coefficients} that
    \begin{align*}
        0 = R^{-\top} \Big[\mathbb{M} \mathfrak{g}_{j+1} - \mathbb{S} \mathfrak{g}_j\Big] \Rightarrow R^{-\top} \mathbb{M} \mathfrak{g}_{j+1} = R^{-\top} \mathbb{S} \mathfrak{g}_j
    \end{align*}
    and applying \eqref{M=R^T R}, we get
    \begin{align*}
        R \mathfrak{g}_{j+1} = R^{-\top} \mathbb{S} \mathfrak{g}_j &\Rightarrow v_{j+1}^{\mathrm{ROM}} = R^{-\top} \mathbb{S} R^{-1} R \mathfrak{g}_{j} \\
        &\Rightarrow v_{j+1}^{\mathrm{ROM}} = R^{-\top} \mathbb{S} R^{-1} v_{j}^{\mathrm{ROM}} =: \mathcal{P}^{\mathrm{ROM}} v_{j}^{\mathrm{ROM}},
    \end{align*}
    which gives the desired result.
\end{proof}

\begin{remark}
    From \eqref{data driven mass matrix} and \eqref{data driven stiffness matrix}, we see that both $R$ and $\mathbb{M}$ are data driven and so, the ROM can be computed directly from $\mathbb{D}(t_j)$, for $j=0,\dots,n$.
\end{remark}

\section{Reduction to scalar wave equation}\label{sec:Scalar}

In this section, we show how the full three-dimensional Maxwell system introduced in Section \ref{sec:Problem Formulation} can be reduced to a one-dimensional scalar wave equation under appropriate assumptions on the geometry and polarization of the fields. This reduction provides the connection between the electromagnetic formulation and the one-dimensional layered media model considered later in this work.

In the reduced-order modeling framework introduced in Section \ref{sec:Problem Formulation}, the effect of the source term can be incorporated into an equivalent initial condition using a Duhamel-type representation. Consequently, in the following, we consider the homogeneous Maxwell system
\begin{align}
    (\partial_t+\mathbb A)\mathbb U=0,
\end{align}
with initial data determined by the excitation. This allows us to focus on the propagation properties of the wave field.

Let $\Omega\subset\mathbb R^3$. In this section we specialize to a layered isotropic medium, with scalar coefficients depending only on one spatial variable:
\begin{align}\label{assumpt 1}
    \varepsilon(x)=\varepsilon(x_1),
    \qquad
    \mu(x)=\mu(x_1).
\end{align}
This corresponds to a medium composed of planar layers orthogonal to the $x_1$-axis. We further consider transverse electromagnetic fields that depend only on $x_1$ and are polarized orthogonally to the direction of variation:
\begin{align}\label{assumpt 2}
    \mathbf E(t,x)
    =
    \begin{pmatrix}
        0\\
        E_2(t,x_1)\\
        0
    \end{pmatrix},
    \qquad
    \mathbf B(t,x)
    =
    \begin{pmatrix}
        0\\
        0\\
        B_3(t,x_1)
    \end{pmatrix}.
\end{align}

\begin{theorem}\label{thm:scalar}
    Under the assumptions \eqref{assumpt 1}-\eqref{assumpt 2}, the source-free Maxwell equations \eqref{eq:Maxwell} reduce to the one-dimensional first-order system
    \begin{align}\label{eq:scalar_first_order_maxwell}
        \left\{
        \begin{aligned}
            \frac{\partial B_3}{\partial t}
            &=
            -\frac{\partial E_2}{\partial x_1},
            \\
            \varepsilon(x_1)\frac{\partial E_2}{\partial t}
            &=
            -
            \frac{\partial}{\partial x_1}
            \left(
            \mu(x_1)^{-1}B_3
            \right).
        \end{aligned}
        \right.
    \end{align}
    Consequently, $E_2$ satisfies the scalar wave equation in divergence form
    \begin{align}\label{eq:E_scalar_variable_mu}
        \varepsilon(x_1)
        \frac{\partial^2 E_2}{\partial t^2}
        -
        \frac{\partial}{\partial x_1}
        \left(
        \mu(x_1)^{-1}
        \frac{\partial E_2}{\partial x_1}
        \right)
        =
        0.
    \end{align}
    Equivalently,
    \begin{align}
        \frac{\partial^2 E_2}{\partial t^2}
        -
        \varepsilon(x_1)^{-1}
        \frac{\partial}{\partial x_1}
        \left(
        \mu(x_1)^{-1}
        \frac{\partial E_2}{\partial x_1}
        \right)
        =
        0.
    \end{align}
    The magnetic flux component satisfies
    \begin{align}\label{eq:B_scalar_variable_mu}
        \frac{\partial^2 B_3}{\partial t^2}
        =
        \frac{\partial}{\partial x_1}
        \left[
        \varepsilon(x_1)^{-1}
        \frac{\partial}{\partial x_1}
        \left(
        \mu(x_1)^{-1}B_3
        \right)
        \right].
    \end{align}
    In a layer where $\varepsilon$ and $\mu$ are constant, \eqref{eq:E_scalar_variable_mu} reduces to
    \begin{align}
        \frac{\partial^2 E_2}{\partial t^2}
        -
        c(x_1)^2
        \frac{\partial^2 E_2}{\partial x_1^2}
        =
        0,
        \qquad
        c(x_1)=\frac{1}{\sqrt{\varepsilon(x_1)\mu(x_1)}}.
    \end{align}
\end{theorem}

\begin{proof}
    See Appendix \ref{app:sec2:thm:scalar}.
\end{proof}

Assuming time-harmonic solutions of the form
\begin{align*}
    E_2(t,x_1)=u(x_1)e^{-i\omega t},
\end{align*}
we obtain the one-dimensional Helmholtz equation
\begin{align}\label{eq:helmholtz_variable_mu}
    \frac{\upd}{\upd x_1}
    \left(
    \mu(x_1)^{-1}
    \frac{\upd u}{\upd x_1}
    \right)
    +
    \omega^2\varepsilon(x_1)u
    =
    0.
\end{align}
Inside a homogeneous layer, where $\varepsilon$ and $\mu$ are constant, this becomes
\begin{align}
    \frac{\upd^2 u}{\upd x_1^2}
    +
    \omega^2\mu(x_1)\varepsilon(x_1)u
    =
    0,
\end{align}
or equivalently
\begin{align}
    \frac{\upd^2 u}{\upd x_1^2}
    +
    \frac{\omega^2}{c(x_1)^2}u
    =
    0.
\end{align}
This reduction leads to the class of one-dimensional layered media, where the wave speed is piecewise constant:
\begin{align}
    c(x_1)=c_k,
    \qquad
    x_1\in(z_k,z_{k+1}),
    \quad k\in\mathbb Z.
\end{align}
At interfaces, the Maxwell transmission conditions correspond to continuity of the tangential electric field and of the tangential magnetic field. In the scalar formulation, this gives continuity of $u$ and of $\mu^{-1}u'$ across interfaces. In each layer, the scalar wave equation admits explicit solutions, with reflection and transmission generated by these interface conditions. A particularly important example is the Goupillaud medium, in which layer thicknesses are proportional to the local wave speed. This model provides a convenient framework for both analytical calculations and numerical simulations.

\section{One-dimensional layered media}\label{sec:Goup}

Layered media provide a classical setting in which multiple scattering and wave reflections can be described explicitly \cite{fouque2007wave}. Related reduced-order approaches in layered media include \cite{borcea2021reduced}. In Section \ref{sec:Scalar}, we showed that, under suitable assumptions on the polarization and spatial variability of the fields, Maxwell's equations reduce to a one-dimensional scalar wave equation with variable electromagnetic coefficients. We now specialize this reduced Maxwell model to layered media and introduce an electromagnetic analogue of the Goupillaud construction. The goal is to obtain a benchmark in which propagation, reflection, transmission and multiple scattering are represented explicitly, while remaining directly connected to the electromagnetic formulation.

The classical Goupillaud model is formulated for acoustic waves in layered media. Here we adapt the same idea to the reduced Maxwell setting by working with electromagnetic impedance variables and imposing an equal travel-time condition across layers. In these variables, the right- and left-going wave amplitudes satisfy explicit interface scattering relations, and the equal travel-time assumption leads to an exact discrete-time propagation model. Thus, the model considered below should be understood as a Maxwell-adapted Goupillaud medium: it preserves the ROM-relevant mechanisms of two-way propagation, reflection and transmission at interfaces, multiple scattering, and data-driven reduced dynamics.

\subsection{Layered medium and wave propagation}

We consider the scalar wave equation obtained in Section \ref{sec:Scalar}, written here in the form
\begin{align}\label{wave eq}
    \frac{\partial^2}{\partial t^2} u(t,z) - c(z)^2 \frac{\partial^2}{\partial z^2} u(t,z) = 0,
\end{align}
where the wave speed $c(z)$ is piecewise constant. This corresponds to a layered medium composed of layers in which the physical parameters are constant. To make the connection with the underlying physical model explicit, we recall that this equation can be derived from the one-dimensional acoustic system
\begin{align}
    \left\{
    \begin{aligned}
        \rho(z)\partial_t u + \partial_z p &= 0, \\
        \frac{1}{K(z)}\partial_t p + \partial_z u &= 0,
    \end{aligned}
    \right. 
\end{align}
where $\rho(z)$ is the density and $K(z)$ the bulk modulus. The wave speed and impedance are given by
\begin{align}
    c(z) = \sqrt{\frac{K(z)}{\rho(z)}} \quad \text{and} \quad \zeta(z) = \sqrt{K(z)\rho(z)}.
\end{align}
In particular, we observe for the impedance $\z(\cdot)$ that it can be written as follows:
\begin{align}
    \zeta(z) = c(z)\rho(z)
\end{align}
Introducing the right- and left-going modes
\begin{align}\label{right and left}
    \left\{
    \begin{aligned}
        A(t,z) &= \zeta(z)^{-1/2}p(t,z) + \zeta(z)^{1/2}u(t,z),\\
        B(t,z) &= -\zeta(z)^{-1/2}p(t,z) + \zeta(z)^{1/2}u(t,z),
    \end{aligned}
    \right.
\end{align}
we obtain the transport equations
\begin{align}
    \partial_z A(t,z) + \frac{1}{c(z)}\partial_t A(t,z) = 0 \quad \text{and} \quad \partial_z B(t,z) - \frac{1}{c(z)}\partial_t B(t,z) = 0.
\end{align}
In addition, since our medium is composed of $n$ layers, we have that
\begin{equation}
    c(z) = c_k \quad \text{ and } \quad \z(z) = \z_k = \rho_k c_k, \quad z \in (z_{k-1}, z_k)
\end{equation}
and so, we can write the solutions of the wave equation \eqref{wave eq} in the following form:
\begin{equation}
    u_k(t,z) = A_k\left(t - \frac{z}{c_k}\right) + B_k\left(t + \frac{z}{c_k}\right),
\end{equation}
where $A_k$ represents right-going waves and $B_k$ represents left-going waves, as described in \eqref{right and left}, for $k=1,2,\dots,n$. At each interface $x=x_k$, continuity conditions give rise to reflection and transmission of the right- and left-going waves. In the impedance-normalized variables used here, the corresponding coefficients are
\begin{align}
    R_k &= \frac{\z_k-\z_{k+1}}{\z_k+\z_{k+1}} \quad \text{ and } \quad
    T_k = \frac{2\sqrt{\z_k \z_{k+1}}}{\z_k+\z_{k+1}}.
\end{align}
for $k=1,2,\dots,n$.

\subsection{Goupillaud medium}

We now introduce a particular class of layered media, known as Goupillaud media, defined by the condition that each layer has the same travel time:
\begin{align}
    \frac{L_j - L_{j-1}}{c_j} = \tau, \quad \text{for all } j.
\end{align}
This assumption plays a crucial role as it implies that waves propagate from one interface to the next in a fixed time $\tau>0$. As a consequence, scattering events occur only at discrete times $t = k\tau$, $k \in \mathbb{N}$, which induces a natural discrete-time dynamical system. Between these discrete times, the right- and left-going waves propagate freely, while reflections and transmissions occur instantaneously at the interfaces. The wave field at time $t = k\tau$ can therefore be written as
\begin{align}
    v_k(z) = \sum_{j=0}^N \left(A_j^k + B_j^k\right)\,\delta(z - L_j),
\end{align}
where $A_j^k$ and $B_j^k$ represent the amplitudes of right- and left-going waves at interface $L_j$. These amplitudes satisfy the recurrence relation
\begin{align}
    \begin{pmatrix}
        A_j^{k+1} \\ B_j^{k+1}
    \end{pmatrix}
    =
    \begin{pmatrix}
        T_j & -R_j \\
        R_j & T_j
    \end{pmatrix}
    \begin{pmatrix}
        A_{j-1}^k \\ B_{j+1}^k
    \end{pmatrix},
\end{align}
where $R_j$ and $T_j$ are the reflection and transmission coefficients.

\begin{figure}[ht]
\centering
\includegraphics[width=0.85\textwidth]{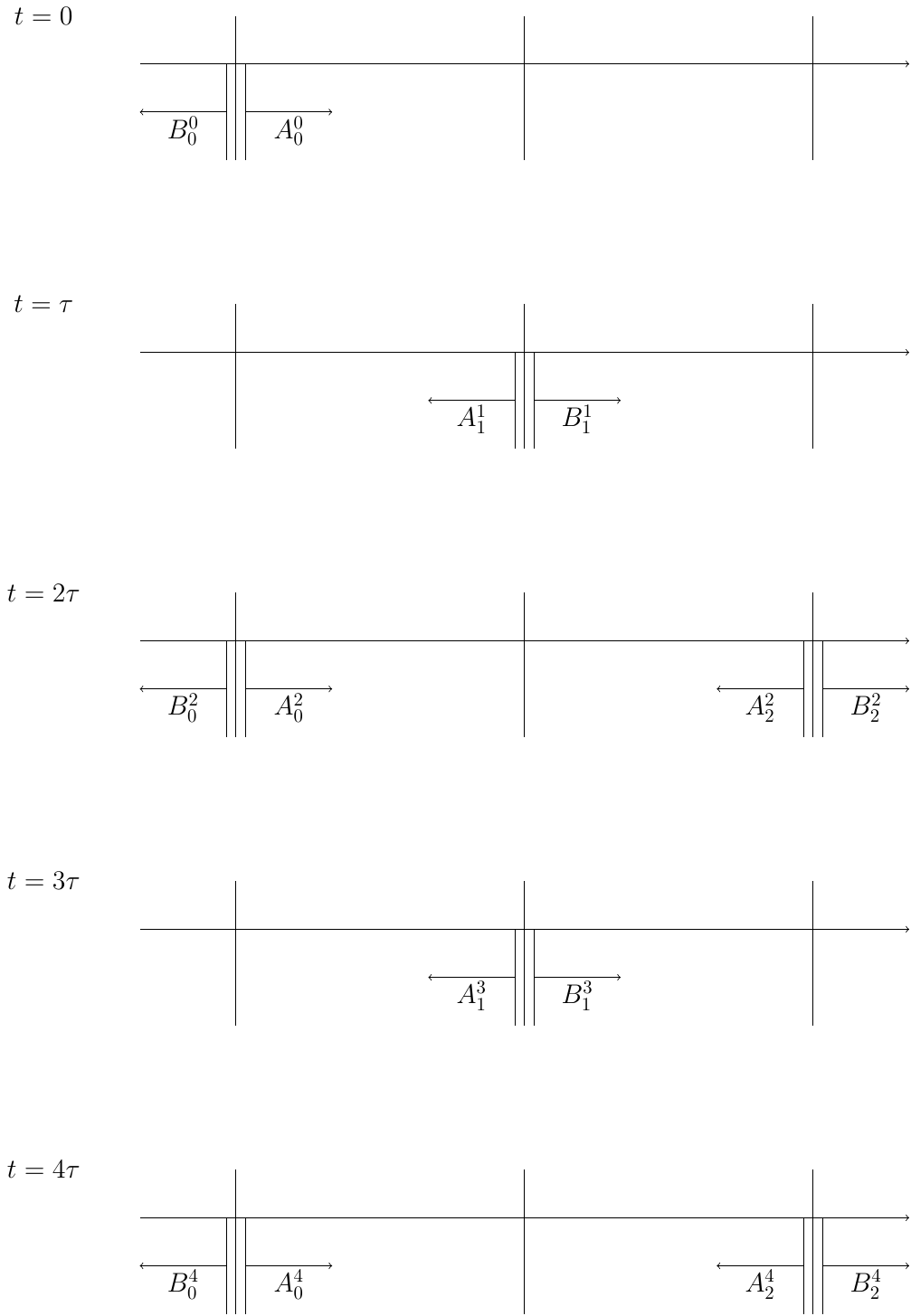}
\caption{
Propagation in a Goupillaud medium. The equal travel-time condition ensures that waves move from one interface to the next in a single time step $\tau$, leading to a discrete dynamical system for the amplitudes $A_j^k$ and $B_j^k$.
}
\end{figure}

\subsection{Connection to ROM formulation}

We consider measurements of the wave field at discrete times $t = k\tau$. The sequence of states $(v_k)_{k=1}^{n}$ defines a collection of snapshots of the system. The observed data are given by
\begin{align}
    D(k\tau) = \langle v_0, v_k \rangle,
\end{align}
as also described by \eqref{eq: D}, which depends on the reflection and transmission coefficients through multiple scattering interactions. Each data point corresponds to a superposition of contributions associated with different scattering paths. The discrete evolution of the wave field takes the form
\begin{align}
    v_{k+1} = \mathcal{P} v_k,
\end{align}
where $\mathcal{P}$ is the propagator over one time step. This structure matches exactly the reduced-order modeling framework introduced in Section \ref{sec:Problem Formulation}. In particular, from Lemma \ref{lemma:M and S properties}, we see that the reduced-order model can be constructed directly from the observed data.

This explicit connection between wave propagation, discrete dynamics, and data-driven reduced-order modeling makes the layered medium an ideal test case for analyzing the stability and performance of ROM-based inversion methods.

\subsection{Inverse problem and objective functions}

We now turn to the inverse problem associated with the layered medium. The goal is to recover the impedance profile
\begin{align}
    \boldsymbol{\z} = (\z_1, \z_2, \dots, \z_n)
\end{align}
from measurements of the reflected wave.

\subsubsection{Forward map}

As described above, the wave propagation through the layered medium gives rise to a sequence of observations
\begin{align}
    \mathbb{D}_k = \mathbb{D}(k\t) = \langle v_0, v_k \rangle
\end{align}
which depend on the impedance parameters through multiple scattering effects. This defines a nonlinear forward map
\begin{align}
    \boldsymbol{\z} \mapsto \mathbb{D}(\boldsymbol{\z}).
\end{align}
This map is highly nonlinear since the waves interact with multiple interfaces, which lead to products of reflection and transmission coefficients.

A natural approach consists of minimizing the difference between modeled and observed data:
\begin{align}
    \mathscr{O}^{\mathrm{FWI}}(\tilde{\boldsymbol{\zeta}}) = \t \sum_{k=1}^{n-1}\| \mathbb{D}(k\t;\boldsymbol{\zeta}) - \mathbb{D}(k\t;\tilde{\boldsymbol{\zeta}}) \|^2.
\end{align}
This objective directly compares the data in the observation space. However, it is sensitive to amplitude errors and scaling effects, which may arise from modeling inaccuracies or measurement noise.

\begin{remark}
    The quadratic data misfit $O^{FWI}$ can be interpreted in a Bayesian framework as a maximum a posteriori estimator under the assumption of additive white Gaussian noise on the data. In this setting, minimizing the $L^2$ discrepancy between observed and modeled data is statistically optimal. This observation explains why the classical misfit is expected to perform well when the perturbations are predominantly unstructured and additive.    
\end{remark}

\subsubsection{Inverse problem}

We formulate the inverse problem as an optimization problem. Given observed data corresponding to an unknown impedance profile $\tilde{\boldsymbol{\z}}$, we seek to recover $\boldsymbol{\z}$ by solving
\begin{align}
    \min_{\boldsymbol{\z}\in\mathcal{Z}} \; \mathscr{O}(\boldsymbol{\z}) + \text{regularization}.
\end{align}
Following the reduced-order modeling framework, the inversion can be formulated in terms of the internal wavefield. Denoting by $v_k(\boldsymbol{\z})$ the true wavefield snapshots and by $v_k^{\mathrm{est}}(\boldsymbol{\z})$ their ROM-based reconstruction, we define the objective
\begin{align}
    \mathscr{O}(\boldsymbol{\z}) =
    \sum_{k=0}^{n_t-1}
    \left\| v_k^{\mathrm{est}}(\boldsymbol{\z}) - v_k(\boldsymbol{\z}) \right\|^2_F.
\end{align}
As in the general ROM framework, this objective can be expressed in terms of a reduced operator. In particular, one obtains
\begin{align}
    \mathscr{O}(\boldsymbol{\z}) = \| R - R(\boldsymbol{\z}) \|_F^2,
\end{align}
where $R$ is the reduced operator constructed from the observed data and $R(\boldsymbol{\z})$ is the one associated with the model.\\
This normalized operator comparison is motivated by ROM-based waveform inversion methods, where the inverse problem is reformulated at the level of reduced propagators or reduced operators rather than directly at the level of the measured traces \cite{borcea2020reduced, borcea2022reduced, borcea2023waveform, borcea20232waveform, borcea2024data}. Instead of directly minimizing $\mathscr{O}$, we consider the modified objective 
\begin{align}
    \min_{\boldsymbol{\z}\in\mathcal{Z}} \; \mathscr{O}^{ROM}(\boldsymbol{\z}) + \text{regularization},
\end{align}
where
\begin{align}
    \mathscr{O}^{ROM}(\boldsymbol{\z}) :=
    \left\| R(\boldsymbol{\z}) R^{-1} - \mathrm{I}_{2nm} \right\|_F^2.
\end{align}

\begin{remark}
    The objective $\mathscr{O}$ measures discrepancies between reconstructed and true internal wavefields through their reduced representation, while $\mathscr{O}^{ROM}$ compares the associated reduced operators. The latter formulation removes global scaling effects and leads to improved stability properties, as will be demonstrated in the numerical experiments.
\end{remark}

\section{Numerical protocol and stability metrics}\label{sec:StabROM}

We now describe the numerical protocol used to compare the direct data misfit and the ROM-based objective. The goal of this section is to specify the benchmark medium, the perturbation models and the evaluation metrics before presenting the numerical results in Section \ref{sec:NumExp}. Further implementation details, including the exact construction of the data vector, the ROM matrix and the perturbation grids, are given in Appendix \ref{app:numerics}.

\subsection{Benchmark configuration}

All numerical experiments are performed on a five-layer Goupillaud medium. The reference impedance vector is
\begin{align}
    \boldsymbol{\zeta}^{\star}
    =
    (1,2,4,8,10,12).
\end{align}
The exterior impedance $\zeta_0=1$ is kept fixed, and the unknown parameter is
\begin{align}
    \theta=(\zeta_1,\ldots,\zeta_5).
\end{align}
For each impedance vector $\boldsymbol{\zeta}$, the Goupillaud recurrence generates a discrete data vector
\begin{align}
    D(\boldsymbol{\zeta})
    =
    (D_0,D_1,\ldots,D_8),
\end{align}
where the odd entries vanish by the parity structure of the model. The corresponding ROM matrix is denoted by $R(\boldsymbol{\zeta})$.

\subsection{Compared inversion objectives}

We compare two inversion strategies. The first one is the direct data misfit
\begin{align}
    \mathscr O^{\mathrm{FWI}}(\boldsymbol{\zeta})
    =
    \tau
    \left\|
    D(\boldsymbol{\zeta})-D^{\mathrm{obs}}
    \right\|_2^2.
\end{align}
The second one is the normalized ROM objective
\begin{align}
    \mathscr O^{\mathrm{ROM}}(\boldsymbol{\zeta}) =
    \left\| R(\boldsymbol{\zeta})R(\boldsymbol{\zeta}_{\rm ref})^\dagger_\alpha - I_5 \right\|_F^2,
\end{align}
where, since $R$ is rectangular, we use the Tikhonov-regularized right inverse
\begin{align}
    R^\dagger_\alpha
    :=
    R^\top(RR^\top+\alpha I)^{-1}.
\end{align}
Here $\boldsymbol{\zeta}_{\rm ref}$ denotes the reference impedance used to build the ROM normalization. In the clean experiment, this is $\boldsymbol{\zeta}^{\star}$ and in the layer-perturbation experiment, this is the perturbed medium generating the data. In perturbation models that act directly on the data and may not correspond exactly to an admissible layered medium, an effective reference impedance is first fitted from the perturbed data. The precise numerical construction is described in Appendix \ref{app:numerics}. This form is a finite-dimensional analogue of the normalized ROM objectives used in \cite{borcea2023waveform, borcea20232waveform, borcea2024data}.

\subsection{Perturbation models}

We consider several structured perturbation models. These perturbations are intended to represent different sources of uncertainty in either the medium or the measurements.

First, we consider the clean inverse problem, where the observed data are generated exactly from $\boldsymbol{\zeta}^{\star}$. This experiment isolates the effect of the optimization landscape, without any additional perturbation.

Second, we perturb the layer impedances themselves. In this case, the perturbed medium is the actual ground truth for the trial, and the goal is to recover this perturbed layered structure.

Third, we consider perturbations acting directly on the data. These include multiplicative data noise, propagation-speed uncertainty, receiver-position errors and systematic acquisition errors. Propagation-speed uncertainty and receiver-position errors are modeled as structured perturbations whose effect can grow with the arrival time. Systematic errors include global gain errors, additive bias and time-origin shifts.

The purpose of these experiments is not to show that the ROM objective is uniformly better for every possible perturbation. Rather, we aim to identify the regimes in which the normalization built into the ROM formulation gives a genuine advantage, and to distinguish them from regimes where the two objectives behave similarly.

\subsection{Evaluation metrics}

For each experiment, we compare the reconstructed impedances obtained from the two objectives. We denote these reconstructions by
\begin{align}
    \widehat{\boldsymbol{\zeta}}_{\mathrm{FWI}},
    \qquad
    \widehat{\boldsymbol{\zeta}}_{\mathrm{ROM}}.
\end{align}
The reconstruction error is measured with respect to a comparison impedance $\boldsymbol{\zeta}_{\rm comp}$:
\begin{align}
    \mathrm{err}_{\mathrm{FWI}}
    &=
    \left\|
    \widehat{\boldsymbol{\zeta}}_{\mathrm{FWI}}
    -
    \boldsymbol{\zeta}_{\rm comp}
    \right\|_2,
    \\
    \mathrm{err}_{\mathrm{ROM}}
    &=
    \left\|
    \widehat{\boldsymbol{\zeta}}_{\mathrm{ROM}}
    -
    \boldsymbol{\zeta}_{\rm comp}
    \right\|_2.
\end{align}
In the clean experiment, $\boldsymbol{\zeta}_{\rm comp}=\boldsymbol{\zeta}^{\star}$. In the layer-perturbation experiment, $\boldsymbol{\zeta}_{\rm comp}$ is the perturbed medium. For data perturbations, the comparison impedance is chosen according to the procedure described in Appendix \ref{app:numerics}.

For each perturbation level, we report the median reconstruction error over the Monte Carlo trials. We also report the win fraction
\begin{align}
    \mathrm{Win}
    =
    \mathbb P
    \left(
    \mathrm{err}_{\mathrm{ROM}}
    <
    \mathrm{err}_{\mathrm{FWI}}
    \right),
\end{align}
estimated empirically over the trials. A win fraction larger than $0.5$ indicates that the ROM-based reconstruction is more accurate than the direct data-misfit reconstruction in a majority of trials.

\section{Numerical results}\label{sec:NumExp}

We now present the numerical results obtained with the protocol described in Section \ref{sec:StabROM}. We focus on the perturbation regimes that are most informative for comparing the two objectives. The remaining perturbation models are summarized at the end of the section.

\subsection{Clean inversion}

We first consider the unperturbed inverse problem. Both methods are initialized from the same family of random initial guesses and are applied to the exact synthetic data generated by $\boldsymbol{\zeta}^{\star}$. Thus, the experiment tests the optimization behavior of the two objectives in the absence of measurement or modeling errors.

In the clean setting, the median reconstruction error is
\begin{align*}
    1.88\times 10^{-3} \quad \text{for the direct data misfit}
    \qquad \text{and} \qquad
    1.34\times 10^{-4} \quad \text{for the ROM-based objective}.
\end{align*}
Thus, the ROM-based objective yields a substantially smaller median reconstruction error than the direct data misfit. Since no perturbation is present in this experiment, this improvement cannot be attributed to noise filtering. It indicates instead that, for this benchmark, the ROM objective has a more favorable numerical landscape for the reconstruction of the impedance profile.

\subsection{Layer perturbations}

We next perturb the layer impedances themselves. In each Monte Carlo trial, the perturbed medium is treated as the ground truth, and both inversion methods are asked to recover it from the corresponding data. This experiment tests robustness with respect to changes in the layered structure, rather than perturbations added only at the measurement level.

\begin{figure}[H]
    \centering
    \begin{subfigure}[t]{0.48\textwidth}
        \centering
        \includegraphics[width=\textwidth]{ 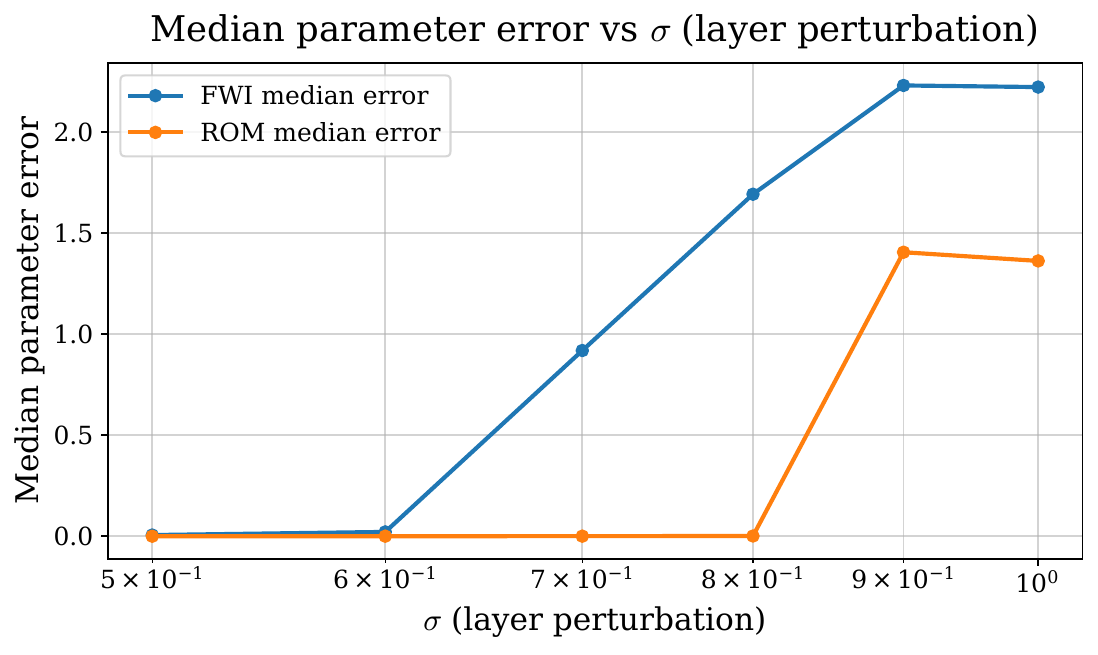}
        \caption{Median reconstruction error.}
    \end{subfigure}\hfill
    \begin{subfigure}[t]{0.48\textwidth}
        \centering
        \includegraphics[width=\textwidth]{ 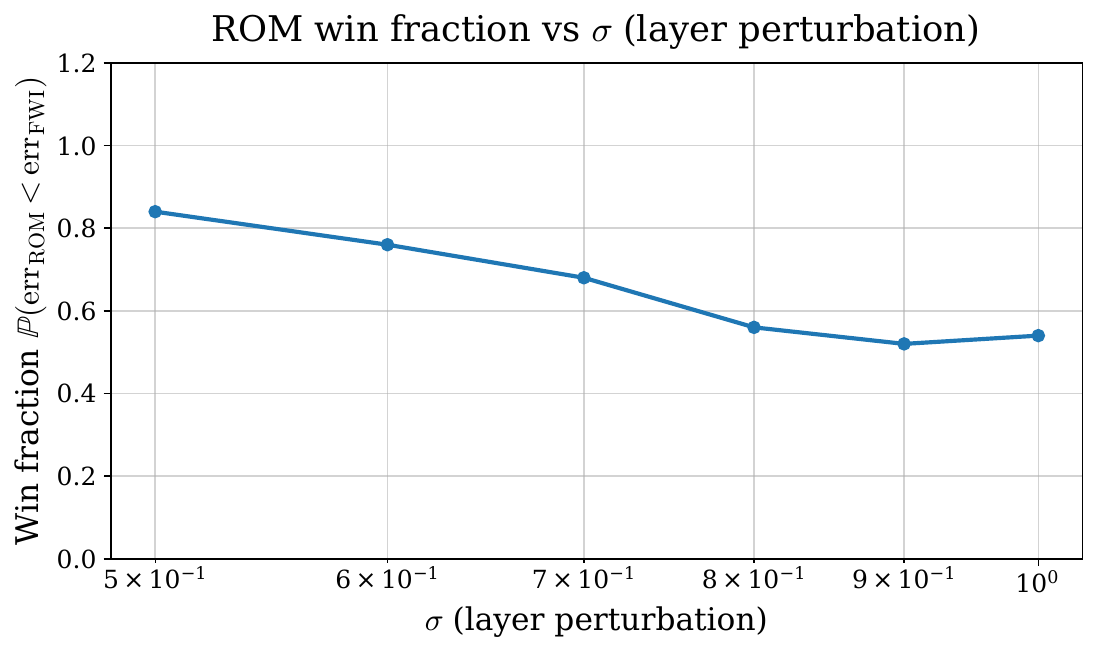}
        \caption{Win fraction $\mathbb{P}(\mathrm{err}_{\mathrm{ROM}}<\mathrm{err}_{\mathrm{FWI}})$.}
    \end{subfigure}
    \caption{Comparison of the two inversion strategies under simultaneous perturbations of all layers.}
    \label{fig:LayerPerturbation}
\end{figure}

Figure \ref{fig:LayerPerturbation}(a) shows that the ROM-based inversion has a smaller median error over the perturbation levels considered. Figure \ref{fig:LayerPerturbation}(b) supports the same conclusion: the win fraction remains above $0.5$, showing that the ROM reconstruction is more accurate in a majority of trials. This suggests that the ROM normalization is beneficial when the uncertainty acts directly on the layered medium.

\subsection{Time-origin error}

We now consider a systematic time-origin error. This perturbation shifts the temporal alignment of the measured data and therefore affects the organization of the arrivals. It is a structured perturbation rather than an independent random perturbation of each data component.

\begin{figure}[H]
    \centering
    \begin{subfigure}[t]{0.48\textwidth}
        \centering
        \includegraphics[width=\textwidth]{ 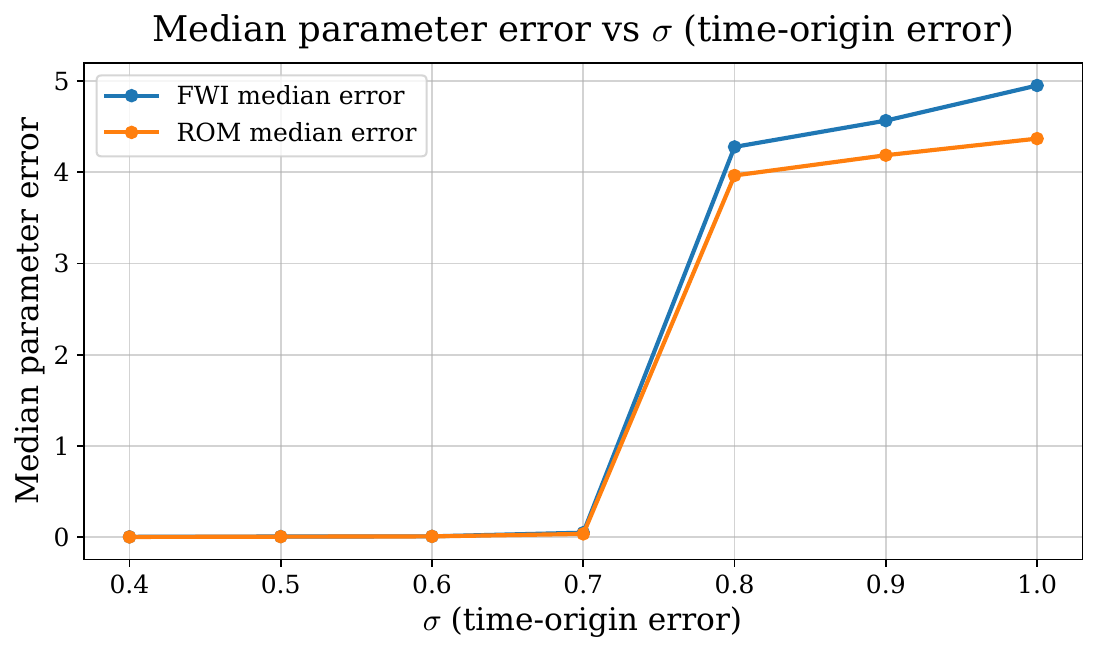}
        \caption{Median reconstruction error.}
    \end{subfigure}\hfill
    \begin{subfigure}[t]{0.48\textwidth}
        \centering
        \includegraphics[width=\textwidth]{ 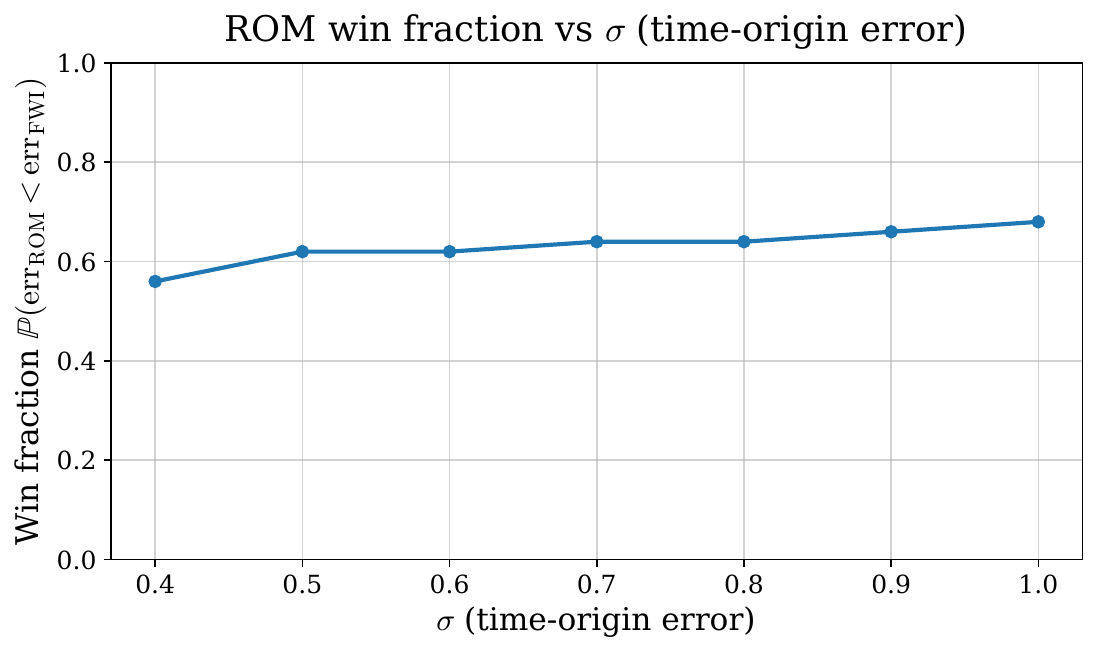}
        \caption{Win fraction $\mathbb{P}(\mathrm{err}_{\mathrm{ROM}}<\mathrm{err}_{\mathrm{FWI}})$.}
    \end{subfigure}
    \caption{Comparison of the two inversion strategies under a time-origin error.}
    \label{fig:TimeOrigin}
\end{figure}

As shown in Figure \ref{fig:TimeOrigin}, the ROM-based inversion gives smaller median errors over most of the tested range. The win fraction is also frequently well above $0.5$. This is one of the clearest regimes in which the ROM objective improves stability. The result is consistent with the fact that the ROM formulation compares normalized reduced operators, and is therefore less sensitive to certain coherent distortions of the data than the direct signal misfit.

\subsection{Receiver-position amplitude perturbation}

We also consider receiver-position errors modeled as amplitude perturbations of the measured arrivals. This type of perturbation represents errors in coupling or geometric spreading that modify the strength of the recorded signal without explicitly shifting its arrival time.

\begin{figure}[H]
    \centering
    \begin{subfigure}[t]{0.48\textwidth}
        \centering
        \includegraphics[width=\textwidth]{ 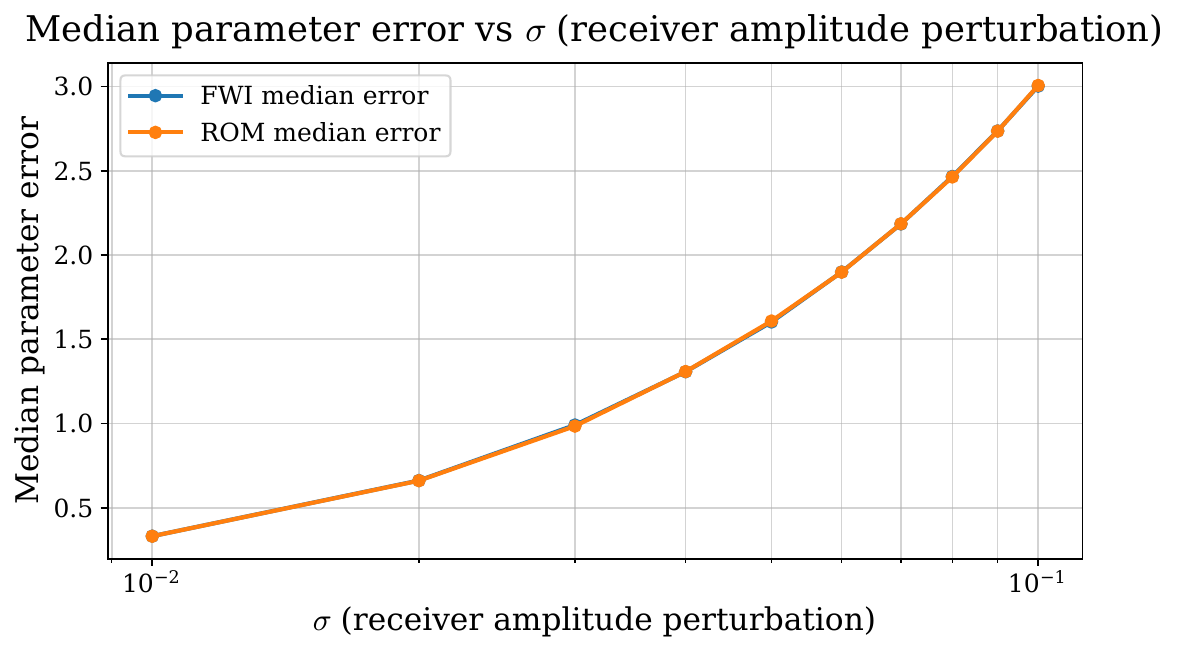}
        \caption{Median reconstruction error.}
    \end{subfigure}\hfill
    \begin{subfigure}[t]{0.48\textwidth}
        \centering
        \includegraphics[width=\textwidth]{ 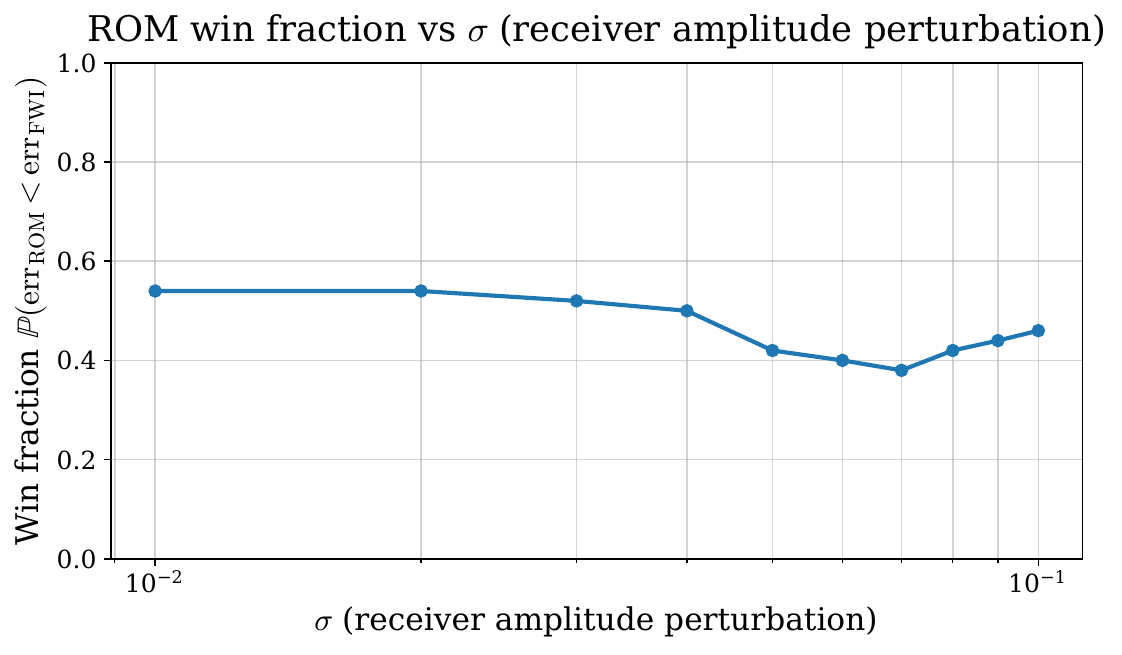}
        \caption{Win fraction $\mathbb{P}(\mathrm{err}_{\mathrm{ROM}}<\mathrm{err}_{\mathrm{FWI}})$.}
    \end{subfigure}
    \caption{Comparison of the two inversion strategies under receiver-position errors modeled as amplitude perturbations.}
    \label{fig:ReceiverAmplitude}
\end{figure}

In Figure \ref{fig:ReceiverAmplitude}, the two methods are closer than in the clean and time-origin experiments. The ROM-based inversion still shows a mild advantage for small to moderate perturbations, visible both in the median errors and in the win fraction, but the improvement is less pronounced. This provides an intermediate regime: the ROM objective remains competitive and slightly preferable, but the perturbation does not produce as clear a separation between the two methods.

\subsection{Perturbations with comparable behavior}

We also tested additional structured perturbations, including multiplicative data noise, propagation-speed uncertainty, receiver-position errors modeled as time shifts, constant gain bias and additive bias. In these cases, the two inversion procedures behave comparably. The median errors remain close, and the win fraction typically fluctuates around $0.5$, without revealing a persistent dominance of either method.

These experiments are important because they show that the ROM objective is not uniformly superior under all perturbations. Its advantage appears most clearly in the clean inversion, in layer perturbations and in coherent temporal distortions such as the time-origin error. For other perturbation mechanisms, the normalization in the ROM objective does not significantly change the reconstruction accuracy, but it also does not lead to a systematic degradation.

\subsection{Summary of numerical observations}

The numerical experiments support a balanced conclusion. In the clean deterministic inverse problem, the ROM-based objective produces a more favorable reconstruction than the direct data misfit. This advantage persists when the layered medium itself is perturbed, and it is especially visible under time-origin errors, where the perturbation has a coherent temporal structure.

At the same time, the ROM formulation does not dominate the direct data misfit in every perturbation regime. For several structured data perturbations, the two methods perform similarly. Thus, the main conclusion is not that ROM-based inversion is universally better, but that it can provide a clear advantage in regimes where the reduced operator normalization captures the relevant structure of the wave propagation more effectively than a direct comparison of the measured signals.

\section{Conclusion}\label{sec:Conclusion}

In this work, we investigated reduced-order modeling for wave-based inverse problems in layered media within the framework of Goupillaud media. This setting provides a natural discrete representation of wave propagation and allows for a direct comparison between classical data misfit and a ROM-based objective defined at the level of reduced operators. In this sense, the numerical study complements the ROM-based inversion program developed in \cite{borcea2020reduced, borcea2022reduced, borcea2023waveform, borcea20232waveform, borcea2024data}, by testing the behavior of a normalized ROM objective in an exactly discretizable layered benchmark.

The numerical experiments highlight a nuanced but consistent picture. In the clean inversion setting, the ROM-based formulation yields a significantly improved reconstruction, indicating a more favorable optimization landscape. This advantage persists in certain structured perturbation regimes, such as perturbations of the layered medium and time-origin errors, where the ROM-based inversion achieves smaller median errors and higher win fractions. These results suggest that the normalization inherent in the ROM formulation is particularly effective in handling global or structurally coherent distortions of the data. At the same time, the ROM-based approach does not uniformly outperform the classical misfit. For several other perturbations, including additive noise and propagation-related distortions, the two methods exhibit comparable performance, with no systematic dominance of either approach. This indicates that the benefit of the ROM formulation depends on the nature of the perturbation and how it interacts with the underlying wave dynamics.

Overall, the results show that reduced-order modeling provides a meaningful and, in some regimes, advantageous alternative to standard inversion approaches. Its strength lies in capturing structural features of the problem that are not directly emphasized by the data misfit, while remaining competitive across a broad range of perturbations.

\appendix

\section{Proofs of Section \ref{sec:Problem Formulation}}

\subsection{Proof of Lemma \ref{lem:hyp prob}} \label{app:sec1:hyp prob}

From \eqref{eq:Maxwell}, we have
\begin{align}
    \frac{\partial \mathbf B}{\partial t}
    =
    -\nabla\times\mathbf E-\mathbf J_m
\end{align}
and
\begin{align}
    \frac{\partial \mathbf E}{\partial t}
    =
    \boldsymbol{\ve}^{-1}
    \nabla\times\bigl(\mu^{-1}\mathbf B\bigr)
    -
    \boldsymbol{\ve}^{-1}\mathbf J_e.
\end{align}
Thus, for
\begin{align*}
    \mathbf U=
    \begin{pmatrix}
        \mathbf E\\
        \mathbf B
    \end{pmatrix},
\end{align*}
we obtain
\begin{align}
    \frac{\partial}{\partial t}
    \begin{pmatrix}
        \mathbf E\\
        \mathbf B
    \end{pmatrix}
    =
    \begin{pmatrix}
        0 & \boldsymbol{\ve}^{-1}\nabla\times\mu^{-1}\\
        -\nabla\times & 0
    \end{pmatrix}
    \begin{pmatrix}
        \mathbf E\\
        \mathbf B
    \end{pmatrix}
    -
    \begin{pmatrix}
        \boldsymbol{\ve}^{-1}\mathbf J_e\\
        \mathbf J_m
    \end{pmatrix}.
\end{align}
Equivalently,
\begin{align}
    \left(\mathbf A+\frac{\partial}{\partial t}\right)\mathbf U
    =
    -
    \begin{pmatrix}
        \boldsymbol{\ve}^{-1}\mathbf J_e\\
        \mathbf J_m
    \end{pmatrix},
\end{align}
where
\begin{align}
    \mathbf A
    =
    \begin{pmatrix}
        0 & -\boldsymbol{\ve}^{-1}\nabla\times\mu^{-1}\\
        \nabla\times & 0
    \end{pmatrix}.
\end{align}
This proves the result.

\subsection{Proof of Lemma \ref{lemma:hyp prob:2}}\label{app:sec1:hyp prob:2}

From \eqref{def: E tilde, B tilde}, we have
\begin{align}
    \mathbf E
    =
    \boldsymbol{\ve}^{-\frac12}\tilde{\mathbf E},
    \qquad
    \mathbf B
    =
    \mu^{\frac12}\tilde{\mathbf B}.
\end{align}
Since $\boldsymbol{\ve}$ and $\mu$ are independent of time,
\begin{align}
    \frac{\partial \tilde{\mathbf E}}{\partial t}
    &=
    \boldsymbol{\ve}^{\frac12}
    \frac{\partial \mathbf E}{\partial t}
    \\
    &=
    \boldsymbol{\ve}^{-\frac12}
    \nabla\times
    \bigl(\mu^{-1}\mathbf B\bigr)
    -
    \boldsymbol{\ve}^{-\frac12}\mathbf J_e
    \\
    &=
    \boldsymbol{\ve}^{-\frac12}
    \nabla\times
    \bigl(\mu^{-\frac12}\tilde{\mathbf B}\bigr)
    -
    \boldsymbol{\ve}^{-\frac12}\mathbf J_e.
\end{align}
Similarly,
\begin{align}
    \frac{\partial \tilde{\mathbf B}}{\partial t}
    &=
    \mu^{-\frac12}
    \frac{\partial \mathbf B}{\partial t}
    \\
    &=
    -
    \mu^{-\frac12}
    \nabla\times
    \bigl(\boldsymbol{\ve}^{-\frac12}\tilde{\mathbf E}\bigr)
    -
    \mu^{-\frac12}\mathbf J_m.
\end{align}
Therefore,
\begin{align}
    \frac{\partial}{\partial t}
    \begin{pmatrix}
        \tilde{\mathbf E}\\
        \tilde{\mathbf B}
    \end{pmatrix}
    =
    \begin{pmatrix}
        0 & \boldsymbol{\ve}^{-\frac12}\nabla\times\mu^{-\frac12}\\
        -\mu^{-\frac12}\nabla\times\boldsymbol{\ve}^{-\frac12} & 0
    \end{pmatrix}
    \begin{pmatrix}
        \tilde{\mathbf E}\\
        \tilde{\mathbf B}
    \end{pmatrix}
    -
    \begin{pmatrix}
        \boldsymbol{\ve}^{-\frac12}\mathbf J_e\\
        \mu^{-\frac12}\mathbf J_m
    \end{pmatrix}.
\end{align}
Thus
\begin{align}
    \left(\tilde{\mathbf A}+\frac{\partial}{\partial t}\right)\tilde{\mathbf U}
    =
    -
    \begin{pmatrix}
        \boldsymbol{\ve}^{-\frac12}\mathbf J_e\\
        \mu^{-\frac12}\mathbf J_m
    \end{pmatrix},
\end{align}
with
\begin{align}
    \tilde{\mathbf A}
    =
    \begin{pmatrix}
        0 & -\boldsymbol{\ve}^{-\frac12}\nabla\times\mu^{-\frac12}\\
        \mu^{-\frac12}\nabla\times\boldsymbol{\ve}^{-\frac12} & 0
    \end{pmatrix}.
\end{align}
This proves the normalized formulation.

\subsection{Proof of Lemma \ref{lem: tilde A skew sym}}\label{app:sec1:tilde A skew sym}

Let
\begin{align*}
    \mathbf u=
    \begin{pmatrix}
        \mathbf u_1\\
        \mathbf u_2
    \end{pmatrix},
    \qquad
    \mathbf v=
    \begin{pmatrix}
        \mathbf v_1\\
        \mathbf v_2
    \end{pmatrix}
\end{align*}
belong to the domain of $\tilde{\mathbf A}$. The boundary condition \eqref{boundary conditions physical} implies, in the normalized variables, that
\begin{align}\label{boundary conditions normalized}
    \mathbf n\times
    \bigl(\boldsymbol{\ve}^{-\frac12}\mathbf u_1\bigr)
    =
    0,
    \qquad
    \mathbf n\times
    \bigl(\boldsymbol{\ve}^{-\frac12}\mathbf v_1\bigr)
    =
    0,
    \qquad x\in\partial\Omega.
\end{align}
Indeed, if \(\mathbf u_1=\widetilde{\mathbf E}\), then
\begin{align*}
    \boldsymbol{\ve}^{-\frac12}\mathbf u_1=\mathbf E.
\end{align*}
Thus \eqref{boundary conditions normalized} is the normalized form of the perfectly conducting condition
\begin{align*}
    \mathbf n\times \mathbf E=0.
\end{align*}
Then,
\begin{align}
    \langle \tilde{\mathbf A}\mathbf u,\mathbf v\rangle
    &=
    -\int_{\W}
    \left[
    \boldsymbol{\ve}^{-\frac12}
    \nabla\times
    \bigl(\mu^{-\frac12}\mathbf u_2\bigr)
    \right]\cdot \mathbf v_1
    \upd x
    +
    \int_{\W}
    \left[
    \mu^{-\frac12}
    \nabla\times
    \bigl(\boldsymbol{\ve}^{-\frac12}\mathbf u_1\bigr)
    \right]\cdot \mathbf v_2
    \upd x.
\end{align}
Equivalently,
\begin{align}
    \langle \tilde{\mathbf A}\mathbf u,\mathbf v\rangle
    &=
    -\int_{\W}
    \nabla\times
    \bigl(\mu^{-\frac12}\mathbf u_2\bigr)
    \cdot
    \boldsymbol{\ve}^{-\frac12}\mathbf v_1
    \upd x
    +
    \int_{\W}
    \nabla\times
    \bigl(\boldsymbol{\ve}^{-\frac12}\mathbf u_1\bigr)
    \cdot
    \mu^{-\frac12}\mathbf v_2
    \upd x.
\end{align}
Using the curl integration-by-parts formula
\begin{align}
    \int_\Omega (\nabla\times a)\cdot b\,\upd x
    =
    \int_\Omega a\cdot(\nabla\times b)\,\upd x
    +
    \int_{\partial\Omega}(\mathbf n\times a)\cdot b\,\upd S,
\end{align}
the boundary terms vanish because of the tangential boundary conditions above. Hence
\begin{align}
    \langle \tilde{\mathbf A}\mathbf u,\mathbf v\rangle
    &=
    -\int_{\W}
    \mu^{-\frac12}\mathbf u_2
    \cdot
    \nabla\times
    \bigl(\boldsymbol{\ve}^{-\frac12}\mathbf v_1\bigr)
    \upd x
    +
    \int_{\W}
    \boldsymbol{\ve}^{-\frac12}\mathbf u_1
    \cdot
    \nabla\times
    \bigl(\mu^{-\frac12}\mathbf v_2\bigr)
    \upd x.
\end{align}
Therefore,
\begin{align}
    \langle \tilde{\mathbf A}\mathbf u,\mathbf v\rangle
    &=
    -\int_{\W}
    \mathbf u_2
    \cdot
    \left[
    \mu^{-\frac12}
    \nabla\times
    \bigl(\boldsymbol{\ve}^{-\frac12}\mathbf v_1\bigr)
    \right]
    \upd x
    +
    \int_{\W}
    \mathbf u_1
    \cdot
    \left[
    \boldsymbol{\ve}^{-\frac12}
    \nabla\times
    \bigl(\mu^{-\frac12}\mathbf v_2\bigr)
    \right]
    \upd x.
\end{align}
By the definition of $\tilde{\mathbf A}$, this is
\begin{align}
    \langle \tilde{\mathbf A}\mathbf u,\mathbf v\rangle
    =
    -
    \langle \mathbf u,\tilde{\mathbf A}\mathbf v\rangle.
\end{align}
Thus $\tilde{\mathbf A}$ is skew-symmetric. The skew-adjointness of the corresponding perfectly conducting Maxwell realization is a standard result; see, for instance, \cite{monk2003finite}.

\subsection{Proof of Theorem \ref{thm:Helmholtz dec}}\label{app:sec1:thm:Helmholtz dec}

Let
\begin{align*}
    \psi=
    \begin{pmatrix}
        \psi_1\\
        \psi_2
    \end{pmatrix}
    \in
    \mathrm{null}(\tilde{\mathbf A}).
\end{align*}
Then $\tilde{\mathbf A}\psi=0$ implies
\begin{align}
    \nabla\times
    \bigl(\mu^{-\frac12}\psi_2\bigr)
    =
    0,
    \qquad
    \nabla\times
    \bigl(\boldsymbol{\ve}^{-\frac12}\psi_1\bigr)
    =
    0.
\end{align}
On a simply connected domain, there exist scalar potentials $N_1$ and $N_2$ such that
\begin{align}
    \boldsymbol{\ve}^{-\frac12}\psi_1
    =
    \nabla N_1,
    \qquad
    \mu^{-\frac12}\psi_2
    =
    \nabla N_2.
\end{align}
Therefore
\begin{align}
    \psi_1
    =
    \boldsymbol{\ve}^{\frac12}\nabla N_1,
    \qquad
    \psi_2
    =
    \mu^{\frac12}\nabla N_2.
\end{align}
This gives
\begin{align}
    \psi
    =
    \widetilde{\nabla}
    \begin{pmatrix}
        N_1\\
        N_2
    \end{pmatrix},
\end{align}
where
\begin{align}
    \widetilde{\nabla}
    \begin{pmatrix}
        N_1\\
        N_2
    \end{pmatrix}
    =
    \begin{pmatrix}
        \boldsymbol{\ve}^{\frac12}\nabla N_1\\
        \mu^{\frac12}\nabla N_2
    \end{pmatrix}.
\end{align}
Using the orthogonal decomposition
\begin{align*}
    L^2(\Omega)^6
    =
    \mathrm{null}(\tilde{\mathbf A})
    \oplus
    \overline{\mathrm{range}(\tilde{\mathbf A})},
\end{align*}
which follows from the general identity
\begin{align*}
    \mathrm{null}(A)=\mathrm{range}(A)^\perp
\end{align*}
for skew-adjoint operators, see for instance \cite{monk2003finite}, the solution decomposes as
\begin{align}
    \tilde{\mathbf{U}}^{(s,p)}(t,x)
    =
    \widetilde{\nabla}\mathbf N^{(s,p)}(t,x)
    +
    \mathbb U^{(s,p)}(t,x).
\end{align}
This proves the decomposition.

\subsection{Claim in spectral decomposition}\label{app:sec2:claim}

We justify the construction of the family $(\boldsymbol{\psi}_j)_{j\geq1}$ associated with the orthonormal eigenbasis $(\boldsymbol{\phi}_j)_{j\geq1}$, and prove the identities
\begin{align}\label{app:eq:A eigs}
    \mathbb{A} \boldsymbol{\phi}_j = \sqrt{\lambda_j}\,\boldsymbol{\psi}_j,
    \qquad
    \mathbb{A} \boldsymbol{\psi}_j = -\sqrt{\lambda_j}\,\boldsymbol{\phi}_j.
\end{align}
Let $\mathcal H_r=\mathrm{null}(\tilde{\mathbf A})^\perp$ be the propagating subspace and let
\begin{align}
    \mathbb A
    =
    \tilde{\mathbf A}\big|_{\mathcal D(\tilde{\mathbf A})\cap\mathcal H_r}.
\end{align}
We work with the standard skew-adjoint Maxwell realization on $\mathcal H_r$. Hence
$-\mathbb A^2$ is positive self-adjoint. Assuming the corresponding Maxwell compactness property, $-\mathbb A^2$ has compact resolvent on $\mathcal H_r$. Therefore there exists an orthonormal eigenbasis $(\boldsymbol\phi_j)_{j\ge1}$ of $\mathcal H_r$ and positive eigenvalues $(\lambda_j)_{j\ge1}$, with $\lambda_j\to\infty$, such that
\begin{align}
    -\mathbb A^2\boldsymbol\phi_j
    =
    \lambda_j\boldsymbol\phi_j.
\end{align}
We do not assume that the eigenvalues $(\lambda_j)_{j\geq1}$ are simple. In general, they may have multiplicity. In that case, the construction above is performed within each eigenspace and an orthonormal basis of eigenfunctions is chosen.

Now, for each $j \geq 1$, we define
\begin{align}
    \boldsymbol{\psi}_j := \frac{\mathbb{A}\boldsymbol{\phi}_j}{\sqrt{\lambda_j}}.
\end{align}
This is well-defined since $\lambda_j > 0$ and $\boldsymbol{\phi}_j \in \mathcal{D}(\mathbb{A})$. We will first show that $(\boldsymbol{\psi}_j)_j$ is an orthonormal family. Using the skew-symmetry of $\mathbb{A}$, we compute
\begin{align*}
    \|\boldsymbol{\psi}_j\|^2 &= \frac{1}{\lambda_j} \langle \mathbb{A}\boldsymbol{\phi}_j, \mathbb{A}\boldsymbol{\phi}_j \rangle = -\frac{1}{\lambda_j} \langle \boldsymbol{\phi}_j, \mathbb{A}^2 \boldsymbol{\phi}_j \rangle = \frac{1}{\lambda_j} \langle \boldsymbol{\phi}_j, -\mathbb{A}^2 \boldsymbol{\phi}_j \rangle \\
    &= \frac{1}{\lambda_j} \lambda_j \langle \boldsymbol{\phi}_j, \boldsymbol{\phi}_j \rangle = 1.
\end{align*}
Thus, $\|\boldsymbol{\psi}_j\| = 1$. Then, let $j,\ell \geq 1$. It holds that
\begin{align*}
    \langle \boldsymbol{\psi}_j, \boldsymbol{\psi}_\ell \rangle &= \frac{1}{\sqrt{\lambda_j \lambda_\ell}} \langle \mathbb{A}\boldsymbol{\phi}_j, \mathbb{A}\boldsymbol{\phi}_\ell \rangle = -\frac{1}{\sqrt{\lambda_j \lambda_\ell}} \langle \boldsymbol{\phi}_j, \mathbb{A}^2 \boldsymbol{\phi}_\ell \rangle = \frac{1}{\sqrt{\lambda_j \lambda_\ell}} \langle \boldsymbol{\phi}_j, -\mathbb{A}^2 \boldsymbol{\phi}_\ell \rangle \\
    &= \frac{\lambda_\ell}{\sqrt{\lambda_j \lambda_\ell}} \langle \boldsymbol{\phi}_j, \boldsymbol{\phi}_\ell \rangle = \sqrt{\frac{\lambda_\ell}{\lambda_j}}\,\delta_{j\ell}.
\end{align*}
Hence,
\begin{align}
    \langle \boldsymbol{\psi}_j, \boldsymbol{\psi}_\ell \rangle = \delta_{j\ell},
\end{align}
which gives the orthogonality.

Let us also show the orthogonality between $\boldsymbol{\psi}_j$ and $\boldsymbol{\phi}_j$, for $j\geq1$. 
Using skew-symmetry,
\begin{align}
    \langle \boldsymbol{\phi}_j, \boldsymbol{\psi}_j \rangle = \frac{1}{\sqrt{\lambda_j}} \langle \boldsymbol{\phi}_j, \mathbb{A}\boldsymbol{\phi}_j \rangle
    = -\frac{1}{\sqrt{\lambda_j}} \langle \mathbb{A}\boldsymbol{\phi}_j, \boldsymbol{\phi}_j \rangle
    = - \langle \boldsymbol{\phi}_j, \boldsymbol{\psi}_j \rangle,
\end{align}
which implies
\begin{align}
    \langle \boldsymbol{\phi}_j, \boldsymbol{\psi}_j \rangle = 0.
\end{align}
In general, the functions $\boldsymbol{\psi}_j$ belong to the span of the family $(\boldsymbol{\phi}_k)_k$, so orthogonality between $\phi_j$ and $\psi_\ell$ does not hold for $\ell \neq j$, with $j,\ell\geq1$.

Finally, let us prove \eqref{app:eq:A eigs}. By definition,
\begin{align}
    \mathbb{A} \boldsymbol{\phi}_j = \sqrt{\lambda_j}\,\boldsymbol{\psi}_j.
\end{align}
Furthermore,
\begin{align}
    \mathbb{A} \boldsymbol{\psi}_j = \frac{1}{\sqrt{\lambda_j}} \mathbb{A}^2 \boldsymbol{\phi}_j = -\frac{\lambda_j}{\sqrt{\lambda_j}} \boldsymbol{\phi}_j = -\sqrt{\lambda_j}\,\boldsymbol{\phi}_j.
\end{align}
This concludes the proof.

\subsection{Proof of Lemma \ref{lemma:M and S properties}} \label{app:sec1:lemma:M and S properties}

\begin{proof}
    For $0\leq j \leq l \leq n-1$, it holds from \eqref{def: mass matrix} that
    \begin{align*}
        \begin{aligned}
            \mathbb{M}_{j,l} &= \int_{\W} v^{\top}_j(x) v_l(x) \upd x = \int_{\W} v_0^{\top}(x) e^{j\t\mathbb{A}} e^{-l\t\mathbb{A}} v_0(x) \upd x \\
            &= \int_{\W} v_0^{\top}(x) e^{-(l-j)\t\mathbb{A}} v_0(x) \upd x,
        \end{aligned}
    \end{align*}
    which, from \eqref{eq: D}, becomes
    \begin{align*}
        \mathbb{M}_{j,l} = \mathbb{D}(t_{l-j}).
    \end{align*}
    In addition, we have from \eqref{M=R^T R}, that the mass matrix $\mathbb{M}$ is symmetric, which implies that for $0\leq l \leq j \leq n-1$, we have
    \begin{align}
        \mathbb{M}_{j,l} = \Big(\mathbb{M}_{l,j}\Big)^{\top}.
    \end{align}
    Similarly, for $0\leq j \leq l \leq n-1$, we have from \eqref{def: stiffness matrix} and \eqref{def: v_j+1} that
    \begin{align*}
        \mathbb{S}_{j,l} &= \int_{\W} v^{\top}_j(x) \mathcal{P} v_l(x) \upd x = \int_{\W} v^{\top}_j(x) v_{l+1}(x) \upd x = \int_{\W} v_0^{\top}(x) e^{j\t\mathbb{A}} e^{-(l+1)\t\mathbb{A}} v_0(x) \upd x \\
        &= \int_{\W}  v_0^{\top}(x) e^{-(l+1-j)\t\mathbb{A}} v_0(x) \upd x,
    \end{align*}
    which gives
    \begin{align*}
        \mathbb{S}_{j,l} = \mathbb{D}(t_{l+1-j}).
    \end{align*}
    Then, for $0\leq l < j-1$, with $j=2,\dots,n-1$, we have
    \begin{align*}
        \mathbb{S}_{j,l} &= \int_{\W} v^{\top}_j(x) \mathcal{P} v_l(x) \upd x = \int_{\W} \Big[\mathcal{P}^{\top } v_j(x)\Big]^{\top} v_l(x) \upd x = \int_{\W} v_{j-1}^{\top}(x) v_l(x) \upd x \\
    &= \Big[\int_{\W} v_l^{\top}(x) v_{j-1}(x) \upd x \Big]^{\top},
    \end{align*}
    which gives
    \begin{align}
        \mathbb{S}_{j,l} = \Big[\mathbb{D}(t_{j-1-l})\Big]^{\top}.
    \end{align}
    This concludes the proof.
\end{proof}

\section{Proofs of Section \ref{sec:Scalar}}

\subsection{Proof of Theorem \ref{thm:scalar}}\label{app:sec2:thm:scalar}

\begin{proof}
    We compute the curl operators explicitly under the assumptions \eqref{assumpt 1}-\eqref{assumpt 2}. Since
    \begin{align*}
        \mathbf E(t,x)
        =
        \begin{pmatrix}
            0\\
            E_2(t,x_1)\\
            0
        \end{pmatrix},
    \end{align*}
    we obtain
    \begin{align}
        \nabla\times\mathbf E
        =
        \begin{pmatrix}
            0\\
            0\\
            \frac{\partial E_2}{\partial x_1}
        \end{pmatrix}.
    \end{align}
    Similarly, since
    \begin{align*}
        \mathbf B(t,x)
        =
        \begin{pmatrix}
            0\\
            0\\
            B_3(t,x_1)
        \end{pmatrix},
    \end{align*}
    and $\mu=\mu(x_1)$, we have
    \begin{align*}
        \mu^{-1}\mathbf B
        =
        \begin{pmatrix}
            0\\
            0\\
            \mu(x_1)^{-1}B_3(t,x_1)
        \end{pmatrix}.
    \end{align*}
    Therefore
    \begin{align}
        \nabla\times(\mu^{-1}\mathbf B)
        =
        \begin{pmatrix}
            0\\
            -
            \frac{\partial}{\partial x_1}
            \left(
            \mu(x_1)^{-1}B_3
            \right)\\
            0
        \end{pmatrix}.
    \end{align}

    Substituting into the source-free Maxwell equations
    \begin{align}
        \nabla\times\mathbf E+\frac{\partial\mathbf B}{\partial t}=0,
        \qquad
        -\nabla\times(\mu^{-1}\mathbf B)
        +
        \varepsilon(x_1)\frac{\partial\mathbf E}{\partial t}=0,
    \end{align}
    we obtain
    \begin{align}
        \left\{
        \begin{aligned}
            \frac{\partial B_3}{\partial t}
            &=
            -
            \frac{\partial E_2}{\partial x_1},
            \\
            \varepsilon(x_1)
            \frac{\partial E_2}{\partial t}
            &=
            -
            \frac{\partial}{\partial x_1}
            \left(
            \mu(x_1)^{-1}B_3
            \right).
        \end{aligned}
        \right.
    \end{align}
    This proves \eqref{eq:scalar_first_order_maxwell}.

    We now derive the second-order equation for $E_2$. Differentiating the second equation with respect to time gives
    \begin{align}
        \varepsilon(x_1)
        \frac{\partial^2 E_2}{\partial t^2}
        =
        -
        \frac{\partial}{\partial x_1}
        \left(
        \mu(x_1)^{-1}
        \frac{\partial B_3}{\partial t}
        \right).
    \end{align}
    Using
    \begin{align*}
        \frac{\partial B_3}{\partial t}
        =
        -
        \frac{\partial E_2}{\partial x_1},
    \end{align*}
    we obtain
    \begin{align}
        \varepsilon(x_1)
        \frac{\partial^2 E_2}{\partial t^2}
        =
        \frac{\partial}{\partial x_1}
        \left(
        \mu(x_1)^{-1}
        \frac{\partial E_2}{\partial x_1}
        \right).
    \end{align}
    Hence
    \begin{align}
        \varepsilon(x_1)
        \frac{\partial^2 E_2}{\partial t^2}
        -
        \frac{\partial}{\partial x_1}
        \left(
        \mu(x_1)^{-1}
        \frac{\partial E_2}{\partial x_1}
        \right)
        =
        0.
    \end{align}

    We also derive the equation for $B_3$. Differentiating the first equation with respect to time gives
    \begin{align}
        \frac{\partial^2 B_3}{\partial t^2}
        =
        -
        \frac{\partial}{\partial x_1}
        \left(
        \frac{\partial E_2}{\partial t}
        \right).
    \end{align}
    From the second equation,
    \begin{align}
        \frac{\partial E_2}{\partial t}
        =
        -
        \varepsilon(x_1)^{-1}
        \frac{\partial}{\partial x_1}
        \left(
        \mu(x_1)^{-1}B_3
        \right).
    \end{align}
    Therefore
    \begin{align}
        \frac{\partial^2 B_3}{\partial t^2}
        =
        \frac{\partial}{\partial x_1}
        \left[
        \varepsilon(x_1)^{-1}
        \frac{\partial}{\partial x_1}
        \left(
        \mu(x_1)^{-1}B_3
        \right)
        \right].
    \end{align}

    Finally, if $\varepsilon$ and $\mu$ are constant inside a layer, the equation for $E_2$ becomes
    \begin{align}
        \frac{\partial^2 E_2}{\partial t^2}
        -
        \frac{1}{\varepsilon\mu}
        \frac{\partial^2 E_2}{\partial x_1^2}
        =
        0.
    \end{align}
    Thus the local wave speed is
    \begin{align}
        c(x_1)
        =
        \frac{1}{\sqrt{\varepsilon(x_1)\mu(x_1)}}.
    \end{align}
    This concludes the proof.
\end{proof}

\section{Details of the numerical implementation}\label{app:numerics}

This appendix describes the numerical procedure used in Section \ref{sec:NumExp}. The purpose is to make the construction of the data, the ROM matrices and the perturbation experiments reproducible.

\subsection{Reference medium and unknown parameters}

All experiments are performed on a five-layer Goupillaud medium. The reference impedance vector is
\begin{align}
    \boldsymbol{\zeta}^{\star}
    =
    (1,2,4,8,10,12).
\end{align}
The first impedance $\zeta_0=1$ is kept fixed throughout the inversion, in order to remove the global scaling ambiguity. The unknown parameter vector is therefore
\begin{align}
    \theta = (\zeta_1,\ldots,\zeta_5).
\end{align}
The admissible set used in the numerical optimization is
\begin{align}
    10^{-6} \leq \zeta_j \leq 50,
    \qquad j=1,\ldots,5.
\end{align}
The time step is normalized to
\begin{align}
    \tau=1.
\end{align}
We use nine snapshots, indexed by $k=0,\ldots,8$.

\subsection{Forward data}

For each impedance vector $\boldsymbol{\zeta}$, the reflection and transmission coefficients at the interfaces are
\begin{align}
    r_j(\boldsymbol{\zeta})
    =
    \frac{\zeta_j-\zeta_{j+1}}{\zeta_j+\zeta_{j+1}},
    \qquad
    t_j(\boldsymbol{\zeta})
    =
    \frac{2\sqrt{\zeta_j\zeta_{j+1}}}{\zeta_j+\zeta_{j+1}},
    \qquad j=0,\ldots,4.
\end{align}
The right- and left-going amplitudes are then propagated by the Goupillaud recurrence described in Section \ref{sec:Goup}. This gives the snapshot coefficients
\begin{align}
    v_k(\boldsymbol{\zeta})
    =
    \sum_j C_{k,j}(\boldsymbol{\zeta})\delta_{L_j},
    \qquad k=0,\ldots,8.
\end{align}
The discrete data vector is computed from the snapshot inner products
\begin{align}
    \mathbb{D}_m(\boldsymbol{\zeta})
    =
    \langle v_0(\boldsymbol{\zeta}),v_m(\boldsymbol{\zeta})\rangle,
    \qquad m=0,\ldots,8.
\end{align}
By the parity structure of the Goupillaud model, the odd entries vanish:
\begin{align}
    \mathbb{D}_1=\mathbb{D}_3=\mathbb{D}_5=\mathbb{D}_7=0.
\end{align}
Thus, the data used in the experiments are
\begin{align}
    \mathbb{D}(\boldsymbol{\zeta})
    =
    (\mathbb{D}_0,0,\mathbb{D}_2,0,\mathbb{D}_4,0,\mathbb{D}_6,0,\mathbb{D}_8).
\end{align}

\subsection{ROM matrix and normalized objective}

The ROM matrix is assembled from the nonzero Goupillaud amplitudes. In the five-layer experiment it has size $5\times 9$ and sparsity pattern
\begin{align}
R(\boldsymbol{\zeta})
=
\begin{pmatrix}
R_{00} & 0 & R_{02} & 0 & R_{04} & 0 & R_{06} & 0 & R_{08}\\
0 & R_{11} & 0 & R_{13} & 0 & R_{15} & 0 & R_{17} & 0\\
0 & 0 & R_{22} & 0 & R_{24} & 0 & R_{26} & 0 & R_{28}\\
0 & 0 & 0 & R_{33} & 0 & R_{35} & 0 & R_{37} & 0\\
0 & 0 & 0 & 0 & R_{44} & 0 & R_{46} & 0 & R_{48}
\end{pmatrix}.
\end{align}
The entries $R_{ij}$ are obtained from the Goupillaud amplitudes computed by the forward recurrence.

Since $R$ is rectangular, we use the Tikhonov-regularized right inverse
\begin{align}
    R^\dagger_\alpha
    =
    R^\top(RR^\top+\alpha I)^{-1},
    \qquad
    \alpha=10^{-8}.
\end{align}
The ROM objective used in the computations is
\begin{align}
    \mathscr O^{\rm ROM}(\boldsymbol{\zeta})
    =
    \left\|
    R(\boldsymbol{\zeta})R(\boldsymbol{\zeta}_{\rm ref})^\dagger_\alpha
    -
    I_5
    \right\|_F^2.
\end{align}
Here $\boldsymbol{\zeta}_{\rm ref}$ denotes the reference impedance for the corresponding experiment. In the clean and layer-perturbation experiments it is the true medium used to generate the data. In data-perturbation experiments, it is the effective impedance fitted from the perturbed data.

The direct data misfit is
\begin{align}
    \mathscr O^{\rm FWI}(\boldsymbol{\zeta})
    =
    \tau
    \left\|
    D(\boldsymbol{\zeta})-D^{\rm obs}
    \right\|_2^2.
\end{align}

\subsection{Optimization procedure}

All minimizations are performed with the L-BFGS-B algorithm under the box constraints
\begin{align}
    \theta_j\in[10^{-6},50],
    \qquad j=1,\ldots,5.
\end{align}
For the clean inversion, we use $20$ random initializations of the form
\begin{align}
    \theta^{(0)}
    =
    \theta^\star\odot(1+0.3\xi),
    \qquad
    \xi\sim\mathcal N(0,I_5),
\end{align}
with clipping to the admissible interval.

For each perturbation level $\sigma$, the reported statistics are computed over $50$ independent Monte Carlo trials. The random seed for the perturbation in trial $k$ is $k$, and the random seed for the initialization is $1000+k$.

\subsection{Perturbation models}

We use the following perturbation models.

\paragraph*{Layer perturbation.}

The physical medium itself is perturbed. We set
\begin{align}
    \zeta_j^{\rm pert}
    =
    \zeta_j^\star(1+\sigma\xi_j),
    \qquad
    j=1,\ldots,5,
\end{align}
where $\xi_j\sim\mathcal N(0,1)$ are independent. The exterior impedance $\zeta_0=1$ is kept fixed. Negative or zero values are clipped to $10^{-6}$. The observed data are
\begin{align}
    \mathbb{D}^{\rm obs}=\mathbb{D}(\boldsymbol{\zeta}^{\rm pert}).
\end{align}
The reconstruction error is measured with respect to $\boldsymbol{\zeta}^{\rm pert}$.

\paragraph*{Multiplicative data noise.}

The physical medium is not changed. The observed data are
\begin{align}
    \mathbb{D}_j^{\rm obs}
    =
    \mathbb{D}_j(\boldsymbol{\zeta}^{\star})(1+\sigma\xi_j),
    \qquad
    \xi_j\sim\mathcal N(0,1).
\end{align}
Since the perturbed data do not necessarily correspond exactly to an admissible layered medium, we first compute an effective reference impedance $\boldsymbol{\zeta}_{\rm ref}$ by minimizing the direct data misfit against $D^{\rm obs}$. The ROM objective is then formed using $R(\boldsymbol{\zeta}_{\rm ref})$, and the reconstruction errors are measured relative to $\boldsymbol{\zeta}_{\rm ref}$.

\paragraph*{Propagation-speed uncertainty.}

Propagation-speed uncertainty is modeled as a multiplicative perturbation of the nonzero even-time data entries, with larger perturbations at later times:
\begin{align}
    \mathbb{D}_j^{\rm obs}
    =
    \mathbb{D}_j(\boldsymbol{\zeta}^{\star})
    \left(1+\sigma\left(1+\frac{j}{8}\right)\xi_j\right),
    \qquad
    j=0,2,4,6,8.
\end{align}
The odd entries are kept equal to zero. The reconstruction error is measured with respect to $\boldsymbol{\zeta}^{\star}$.

\paragraph*{Receiver-position error: time-shift model.}

A receiver-position error is first modeled as a coherent time-shift distortion. A single Gaussian variable $\xi$ is drawn and we set
\begin{align}
    \mathbb{D}_j^{\rm obs}
    =
    \mathbb{D}_j(\boldsymbol{\zeta}^{\star})(1+\sigma\xi)^j,
    \qquad
    j=0,2,4,6,8.
\end{align}
The odd entries are kept equal to zero.

\paragraph*{Receiver-position error: amplitude model.}

We also model receiver-position errors as amplitude perturbations of the recorded arrivals:
\begin{align}
    \mathbb{D}_j^{\rm obs}
    =
    \mathbb{D}_j(\boldsymbol{\zeta}^{\star})(1+\sigma w_j\xi_j),
    \qquad
    j=0,2,4,6,8,
\end{align}
with weights
\begin{align}
    w_0=0,\qquad w_2=1,\qquad w_4=2,\qquad w_6=3,\qquad w_8=4.
\end{align}
Thus the first datum is kept fixed, while later arrivals are perturbed more strongly. The odd entries remain zero.

\paragraph*{Constant gain bias.}

A calibration error is modeled by multiplying the full data vector by a single random gain:
\begin{align}
    \mathbb{D}^{\rm obs}
    =
    (1+\sigma\xi)\mathbb{D}(\boldsymbol{\zeta}^{\star}),
    \qquad
    \xi\sim\mathcal N(0,1).
\end{align}

\paragraph*{Additive bias.}

A baseline error is modeled by adding a constant offset to all entries of the data vector:
\begin{align}
    \mathbb{D}^{\rm obs}
    =
    \mathbb{D}(\boldsymbol{\zeta}^{\star}) + b\mathbf{1},
\end{align}
where
\begin{align}
    b
    =
    \sigma s\xi,
    \qquad
    s=
    \max\left\{
    \frac{\|\mathbb{D}(\boldsymbol{\zeta}^{\star})\|_2}{\sqrt{9}},
    1
    \right\},
    \qquad
    \xi\sim\mathcal N(0,1).
\end{align}

\paragraph*{Time-origin error.}

A time-origin error is modeled by shifting the data vector by an integer number of time steps. We draw
\begin{align}
    k
    =
    \mathrm{clip}
    \left(
    \mathrm{round}(\sigma\xi),
    -2,
    2
    \right),
    \qquad
    \xi\sim\mathcal N(0,1),
\end{align}
and shift $D(\boldsymbol{\zeta}^{\star})$ by $k$ entries, using zero padding at the boundaries. Here $\mathrm{clip}(y,-2,2)$ denotes the projection of $y$ onto the interval $[-2,2]$.

\subsection{Perturbation grids}

The perturbation amplitudes used in the numerical experiments are as follows:
\begin{align}
    \sigma_{\rm layer}
    &\in
    \{0.5,0.6,0.7,0.8,0.9,1.0\},
    \\
    \sigma_{\rm data}
    &\in
    \{10^{-2},2\cdot10^{-2},\ldots,10^{-1}\},
    \\
    \sigma_{\rm speed}
    &\in
    \{10^{-2},2\cdot10^{-2},\ldots,10^{-1}\},
    \\
    \sigma_{\rm rec,time}
    &\in
    \{10^{-2},2\cdot10^{-2},\ldots,10^{-1}\},
    \\
    \sigma_{\rm rec,amp}
    &\in
    \{10^{-2},2\cdot10^{-2},\ldots,10^{-1}\},
    \\
    \sigma_{\rm gain}
    &\in
    \{10^{-2},2\cdot10^{-2},\ldots,10^{-1},2\cdot10^{-1},\ldots,6\cdot10^{-1}\},
    \\
    \sigma_{\rm add}
    &\in
    \{10^{-3},2\cdot10^{-3},\ldots,10^{-2}\},
    \\
    \sigma_{\rm origin}
    &\in
    \{0.4,0.5,0.6,0.7,0.8,0.9,1.0\}.
\end{align}

\subsection{Reported statistics}

For each perturbation level, we compute the reconstruction errors
\begin{align}
    \mathrm{err}_{\rm FWI}
    =
    \|\widehat{\boldsymbol{\zeta}}_{\rm FWI}
    -
    \boldsymbol{\zeta}_{\rm comp}\|_2,
    \qquad
    \mathrm{err}_{\rm ROM}
    =
    \|\widehat{\boldsymbol{\zeta}}_{\rm ROM}
    -
    \boldsymbol{\zeta}_{\rm comp}\|_2.
\end{align}
Here $\boldsymbol{\zeta}_{\rm comp}$ is the comparison impedance: it is the perturbed medium for the layer-perturbation experiment, the fitted effective reference medium for the data-noise experiment, and the original reference medium $\boldsymbol{\zeta}^{\star}$ for the remaining perturbation models.

The figures report the median of these errors over the Monte Carlo trials, together with the win fraction
\begin{align}
    \mathrm{Win}
    =
    \mathbb P
    \left(
    \mathrm{err}_{\rm ROM}
    <
    \mathrm{err}_{\rm FWI}
    \right),
\end{align}
estimated empirically over the $50$ trials.

\bibliographystyle{abbrv}
\bibliography{references}{}

\begin{thebibliography}{10}

\bibitem{belishev2007recent}
M.~I. Belishev.
\newblock Recent progress in the boundary control method.
\newblock {\em Inverse problems}, 23(5):R1--R67, 2007.

\bibitem{benner2015survey}
P.~Benner, S.~Gugercin, and K.~Willcox.
\newblock A survey of projection-based model reduction methods for parametric
  dynamical systems.
\newblock {\em SIAM review}, 57(4):483--531, 2015.

\bibitem{blondel1997handbook}
P.~Blondel and B.~J. Murton.
\newblock {\em Handbook of seafloor sonar imagery}, volume~7.
\newblock Wiley Chichester, 1997.

\bibitem{borcea2018untangling}
L.~Borcea, V.~Druskin, A.~V. Mamonov, and M.~Zaslavsky.
\newblock Untangling the nonlinearity in inverse scattering with data-driven
  reduced order models.
\newblock {\em Inverse Problems}, 34(6):065008, 2018.

\bibitem{borcea2020reduced}
L.~Borcea, V.~Druskin, A.~V. Mamonov, M.~Zaslavsky, and J.~Zimmerling.
\newblock Reduced order model approach to inverse scattering.
\newblock {\em SIAM Journal on Imaging Sciences}, 13(2):685--723, 2020.

\bibitem{borcea2021reduced}
L.~Borcea, V.~Druskin, and J.~Zimmerling.
\newblock A reduced order model approach to inverse scattering in lossy layered
  media.
\newblock {\em Journal of Scientific Computing}, 89(1):1, 2021.

\bibitem{borcea2022reduced}
L.~Borcea, J.~Garnier, A.~V. Mamonov, and J.~Zimmerling.
\newblock Reduced order model approach for imaging with waves.
\newblock {\em Inverse Problems}, 38(2):025004, 2022.

\bibitem{borcea2023waveform}
L.~Borcea, J.~Garnier, A.~V. Mamonov, and J.~Zimmerling.
\newblock Waveform inversion via reduced order modeling.
\newblock {\em Geophysics}, 88(2):R175--R191, 2023.

\bibitem{borcea20232waveform}
L.~Borcea, J.~Garnier, A.~V. Mamonov, and J.~Zimmerling.
\newblock Waveform inversion with a data driven estimate of the internal wave.
\newblock {\em SIAM Journal on Imaging Sciences}, 16(1):280--312, 2023.

\bibitem{borcea2024data}
L.~Borcea, J.~Garnier, A.~V. Mamonov, and J.~Zimmerling.
\newblock When data driven reduced order modeling meets full waveform
  inversion.
\newblock {\em SIAM Review}, 66(3):501--532, 2024.

\bibitem{BORCEA2024113272}
L.~Borcea, Y.~Liu, and J.~Zimmerling.
\newblock Electromagnetic inverse wave scattering in anisotropic media via
  reduced order modeling.
\newblock {\em Journal of Computational Physics}, 515:113272, 2024.

\bibitem{brunton2022data}
S.~L. Brunton and J.~N. Kutz.
\newblock {\em Data-driven science and engineering: Machine learning, dynamical
  systems, and control}.
\newblock Cambridge University Press, 2022.

\bibitem{cheney2009fundamentals}
M.~Cheney and B.~Borden.
\newblock {\em Fundamentals of radar imaging}.
\newblock SIAM, 2009.

\bibitem{curlander1991synthetic}
J.~C. Curlander and R.~N. McDonough.
\newblock {\em Synthetic aperture radar}, volume~11.
\newblock Wiley, New York, 1991.

\bibitem{druskin2016direct}
V.~Druskin, A.~V. Mamonov, A.~E. Thaler, and M.~Zaslavsky.
\newblock Direct, nonlinear inversion algorithm for hyperbolic problems via
  projection-based model reduction.
\newblock {\em SIAM Journal on Imaging Sciences}, 9(2):684--747, 2016.

\bibitem{druskin2018nonlinear}
V.~Druskin, A.~V. Mamonov, and M.~Zaslavsky.
\newblock A nonlinear method for imaging with acoustic waves via reduced order
  model backprojection.
\newblock {\em SIAM Journal on Imaging Sciences}, 11(1):164--196, 2018.

\bibitem{engquist2013application}
B.~Engquist and B.~D. Froese.
\newblock Application of the wasserstein metric to seismic signals.
\newblock {\em arXiv preprint arXiv:1311.4581}, 2013.

\bibitem{engquist2022optimal}
B.~Engquist and Y.~Yang.
\newblock Optimal transport based seismic inversion: Beyond cycle skipping.
\newblock {\em Communications on Pure and Applied Mathematics},
  75(10):2201--2244, 2022.

\bibitem{fouque2007wave}
J.-P. Fouque, J.~Garnier, G.~Papanicolaou, and K.~Solna.
\newblock {\em Wave propagation and time reversal in randomly layered media},
  volume~56.
\newblock Springer Science \& Business Media, 2007.

\bibitem{gilman2017transionospheric}
M.~Gilman, E.~Smith, and S.~Tsynkov.
\newblock {\em Transionospheric synthetic aperture imaging}.
\newblock Springer, 2017.

\bibitem{hesthaven2022reduced}
J.~S. Hesthaven, C.~Pagliantini, and G.~Rozza.
\newblock Reduced basis methods for time-dependent problems.
\newblock {\em Acta Numerica}, 31:265--345, 2022.

\bibitem{huang2018source}
G.~Huang, R.~Nammour, and W.~W. Symes.
\newblock Source-independent extended waveform inversion based on space-time
  source extension: Frequency-domain implementation.
\newblock {\em Geophysics}, 83(5):R449--R461, 2018.

\bibitem{monk2003finite}
P.~Monk.
\newblock {\em Finite element methods for Maxwell's equations}.
\newblock Oxford university press, 2003.

\bibitem{stefanov2005stable}
P.~Stefanov and G.~Uhlmann.
\newblock Stable determination of generic simple metrics from the hyperbolic
  dirichlet-to-neumann map.
\newblock {\em International Mathematics Research Notices},
  2005(17):1047--1061, 2005.

\bibitem{symes2008migration}
W.~W. Symes.
\newblock Migration velocity analysis and waveform inversion.
\newblock {\em Geophysical prospecting}, 56(6):765--790, 2008.

\bibitem{symes2022error}
W.~W. Symes.
\newblock Error bounds for extended source inversion applied to an acoustic
  transmission inverse problem.
\newblock {\em Inverse Problems}, 38(11):115002, 2022.

\bibitem{szabo2013diagnostic}
T.~L. Szabo.
\newblock {\em Diagnostic ultrasound imaging: inside out}.
\newblock Academic press, 2013.

\bibitem{tether1970construction}
A.~Tether.
\newblock Construction of minimal linear state-variable models from finite
  input-output data.
\newblock {\em IEEE Transactions on Automatic Control}, 15(4):427--436, 1970.

\bibitem{van2013mitigating}
T.~Van~Leeuwen and F.~J. Herrmann.
\newblock Mitigating local minima in full-waveform inversion by expanding the
  search space.
\newblock {\em Geophysical Journal International}, 195(1):661--667, 2013.

\bibitem{virieux2010overview}
J.~Virieux and S.~Operto.
\newblock An overview of full-waveform inversion in exploration geophysics.
\newblock {\em Geophysics}, 2010.

\bibitem{warner2016adaptive}
M.~Warner and L.~Guasch.
\newblock Adaptive waveform inversion: Theory.
\newblock {\em Geophysics}, 81(6):R429--R445, 2016.

\bibitem{yang2018application}
Y.~Yang, B.~Engquist, J.~Sun, and B.~F. Hamfeldt.
\newblock Application of optimal transport and the quadratic wasserstein metric
  to full-waveform inversion.
\newblock {\em Geophysics}, 83(1):R43--R62, 2018.

\bibitem{yu2001global}
O.~Yu.~Imanuvilov and M.~Yamamoto.
\newblock Global uniqueness and stability in determining coefficients of wave
  equations.
\newblock {\em Communications in Partial Differential Equations},
  26(7-8):1409--1425, 2001.

\end{thebibliography}
\end{document}